\documentclass[twoside,12pt,leqno]{amsproc}
\usepackage{amssymb,latexsym,enumerate,stmaryrd,tikz-cd,graphics,verbatim}
\usetikzlibrary{patterns}
\usepackage{booktabs}  
\usepackage[pagebackref]{hyperref}
\usepackage{amsrefs}    
\usepackage{etoolbox}   
\usepackage[euler-digits,euler-hat-accent]{eulervm}
\usepackage{mathtools}  
\hypersetup{citecolor=purple, linkcolor=blue, colorlinks=true}
\numberwithin{table}{section}

\theoremstyle{plain}
\newtheorem{theorem}{Theorem}[section]
\newtheorem{lemma}[theorem]{Lemma}
\newtheorem{proposition}[theorem]{Proposition}
\newtheorem{corollary}[theorem]{Corollary}

\theoremstyle{definition} 
\newtheorem{definition}[theorem]{Definition}
\newtheorem{remark}[theorem]{Remark}
\newtheorem{hypothesis}[theorem]{Hypothesis}

\renewcommand{\ge}{\geqslant}
\renewcommand{\le}{\leqslant}
\renewcommand{\geq}{\geqslant}
\renewcommand{\leq}{\leqslant}
\newcommand{\lhdeq}{\trianglelefteqslant}    

\newcommand{\alphadown}{\alpha\hspace{-2.5pt}\downarrow}
\newcommand{\Aut}{\textup{Aut}}
\newcommand{\Char}{\textup{char}}

\newcommand{\F}{\mathbb{F}}
\newcommand\G{\mathcal{G}}
\newcommand\GL{\textup{GL}}
\newcommand\AGL{\textup{AGL}}
\newcommand\im{\textup{im}}

\newcommand\PSL{\textup{PSL}}
\newcommand\SL{\textup{SL}}
\newcommand\Sym{\textup{\textsf{S}}}
\newcommand{\Z}{\textup{\textsf{Z}}}
\newcommand{\UCS}{\mbox{\rm UCS}}
\newcommand{\IAC}{\mbox{\rm IAC}}

\newcommand{\prd}[2]{\llbracket #1,#2\rrbracket}
\newcommand{\iac}[1]{\mathcal L(#1)}
\newcommand{\grp}[1]{\mathcal G(#1)}

\makeatletter        
\def\@adminfootnotes{%
  \let\@makefnmark\relax  \let\@thefnmark\relax
  \ifx\@empty\@date\else \@footnotetext{\@setdate}\fi
  \ifx\@empty\@subjclass\else \@footnotetext{\@setsubjclass}\fi
  \ifx\@empty\@keywords\else \@footnotetext{\@setkeywords}\fi
  \ifx\@empty\thankses\else \@footnotetext{%
    \def\par{\let\par\@par}\@setthanks}%
  \fi}\makeatother   

\begin{document}

\hyphenation{Trans-lated}

\title[UCS $p$-groups and anti-commutative algebras]
      {Duality between $p$-groups with three characteristic subgroups
and semisimple anti-commutative algebras}

\date{\today}

\author{S.\,P. Glasby, Frederico A.\ M. Ribeiro and Csaba Schneider}
\address[Glasby]{
Centre for Mathematics of Symmetry and Computation\\
University of Western Australia\\
35 Stirling Highway\\
Perth 6009, Australia. Email: {\tt Stephen.Glasby@uwa.edu.au}; URL: {\tt\href{http://www.maths.uwa.edu.au/~glasby/}{http://www.maths.uwa.edu.au/$\sim$glasby/}}}
\address[Ribeiro, Schneider]{Departamento de Matem\'atica\\
Instituto de Ci\^encias Exatas\\
Universidade Federal de Minas Gerais\\
Av.\ Ant\^onio Carlos 6627\\
Belo Horizonte, MG, Brazil. Email:
{\rm\texttt{fred321@gmail.com}}, {\rm\texttt{csaba@mat.ufmg.br}}; URL:
{\tt\href{http://www.mat.ufmg.br/~csaba}{http://www.mat.ufmg.br/$\sim$csaba}}}


\begin{abstract}
  Let $p$ be an odd prime and let $G$ be a non-abelian finite $p$-group
  of exponent $p^2$ with
three distinct characteristic subgroups, namely $1$, $G^p$, and $G$.
The quotient group $G/G^p$ gives rise to an anti-commutative
$\F_p$-algebra~$L$ such that the action of $\Aut(L)$ is irreducible on $L$;
we call such an algebra IAC.
This paper establishes a duality $G\leftrightarrow L$
between such groups and such IAC algebras.  We
prove that IAC algebras are semisimple and  we classify the
simple IAC algebras of dimension at most 4 over certain fields.
We also give other examples of simple IAC
algebras, including a family
related to the $m$-th symmetric power of the natural module of
$\SL(2,\F)$.
\end{abstract}

\maketitle
\centerline{\noindent 2010 Mathematics subject classification:
 20D15, 20C20, 20E15, 20F28, 17A30, 17A36}

\section{Introduction}\label{sec1}

Building on earlier work by Taunt~\cite{taunt} and by
the first and the third authors
in collaboration with P.~P.~P\'alfy~\cite{ucs}, 
we continue, in this paper, our investigation into the structure of finite
groups
with a unique non-trivial and proper characteristic
subgroup. Such a group is said to
be UCS. In~\cite{ucs}, finite UCS $p$-groups were studied.
The exponent of a finite UCS $p$-group is either $p$ or $p^2$.
If $p$ is an odd prime and $G$ is a UCS $p$-group of exponent $p^2$, then
the Frattini quotient
$V=G/\Phi(G)$
is an irreducible $\Aut(G)$-module and  $V$ is an epimorphic
image of the exterior square $\wedge^2 V$.
An epimorphism $\wedge^2 V\rightarrow V$ can be viewed as an
anti-commutative multiplication on the vector space $V$ which turns $V$ into
a non-associative and anti-commutative algebra $\iac G$
defined in Section~\ref{sec3}. Further,
by Theorem~\ref{main1}(a) below,
the automorphism group of $\iac G$ is induced by $\Aut(G)$, and is irreducible
on $V$. The main results of this paper explore the connection
between $G$ and $\iac G$.

To state the first main theorem of the paper, let $\UCS_{p^2}$ denote the
set of isomorphism classes of finite UCS $p$-groups of exponent $p^2$ and,
for an odd prime $p$, let $\IAC_p$ denote the set of isomorphism classes of
finite-dimensional anti-commutative algebras $L$ over the field $\F_p$ of $p$ elements such that $\Aut(L)$
is irreducible on $L$. Such an algebra will be referred to as an IAC algebra.

\begin{theorem}\label{main1}
  If $p$ is an odd prime, then the map $G\mapsto \iac G$ is a bijection between $\UCS_{p^2}$ and $\IAC_{p}$.
  Further, if $G\in \UCS_{p^2}$ and $L=\iac G$ then the following hold.
  \begin{enumerate}[{\rm (a)}]
  \item There is an isomorphism $\Aut (G)/\Aut_C(G)\cong \Aut(L)$ where
    $\Aut_C(G)$ denotes the group of central automorphisms of $G$.
    \item There is a bijection between the set of subalgebras of $L$ and the
      set of powerful subgroups of $G$ that contain $\Phi(G)$.
    \item The bijection in part~\textup{(b)} restricts to a bijection between ideals
      of $L$ and powerfully embedded subgroups of $G$ containing $\Phi(G)$.
    \item For $k\geq 2$, we have that $G^{\times k}\in \UCS_{p^2}$ and $\iac{G^{\times k}}\cong
      L^{\oplus k}$.
  \end{enumerate}
\end{theorem}

The first statement of Theorem~\ref{main1} follows from
Theorem~\ref{th:bij}, while
parts (a), (b),  (c) and (d) follow from
Theorem~\ref{th:aut}, Proposition~\ref{pembedded}, and from Theorem~\ref{SumPower}.

In the second part of the paper, we prove several results about
small-dimensional IAC algebras over various fields, which can be summarized in the following theorem.

\begin{theorem}\label{main2}
  The following are valid for a field $\F$. 
  \begin{enumerate}[{\rm (a)}]
  \item A finite-dimensional IAC algebra over $\F$ is the direct sum of
    pairwise isomorphic simple IAC algebras.
  \item If $\Char(\F)\neq 2$, then a non-abelian
    $3$-dimensional IAC algebra over $\F$ is a simple Lie algebra.
  \item Suppose that $\F$ is a finite field of $q=p^f$ elements where $p$
    is an odd prime.  Then the number of isomorphism types of $4$-dimensional IAC
    algebras over $\F$ is \textup{(i)} $0$ if $p=5$, \textup{(ii)} $1$ if $f$
    is even or $p\equiv\pm1\pmod5$,  and \textup{(iii)} $2$ if $f$ is odd and
    $p\equiv \pm 2\pmod 5$.
  \item If $m\geq 2$, $m\equiv 2\pmod 4$, and either $\Char(\F)=0$ or $2m<\Char(\F)$, then there
    is an $(m+1)$-dimensional IAC algebra $L$ over $\F$ whose automorphism
    group contains a subgroup isomorphic to $\SL(2,\F)$ acting
    absolutely irreducibly
    on $L$.
  \end{enumerate}
\end{theorem}

Theorem~\ref{main2} follows from Theorem~\ref{SimpleSum}, Proposition~\ref{prop:51}, Theorem~\ref{th:52}, and Theorem~\ref{T:GammaL}. Note that
$p^f\equiv\pm2\pmod5$ holds if and only if $f$ is odd and $p\equiv\pm2\pmod5$.

UCS $p$-groups are atypical, since
finite $p$-groups usually have a rich structure of characteristic
subgroups, which is used in the optimization of the
standard algorithms for computing $\Aut(G)$; see for instance~\cites{el-gob,wilsonchar}.
However, when $G$ does not have characteristic subgroups apart from the
usual verbal subgroups, such optimization techniques fail.
In particular, the algorithms for
computing $\Aut(G)$ that are based on orbit-stabilizer calculation
would perform poorly on UCS $p$-groups. Nevertheless, Theorem~\ref{main1}
shows that
UCS $p$-groups of exponent $p^2$ have a rigid algebraic structure that,
in theory, can be exploited also in automorphism group computations, in a
way which is rather similar to the  philosophy pursued by
J.~B.~Wilson and
his collaborators in their recent work~\cites{bwtrans,bmwgenus,wilsonauts}.

Following~\cite{ucs}, we say that a $G$-module $V$ is an ESQ module
if $V\cong\wedge^2 V/U$ for some $G$-submodule $U$
(see also Section~\ref{subsec:4dim}).
IAC algebras are related to irreducible ESQ $G$-modules. 
It was observed in~\cite{eqqlie}
that a similar connection exists between Lie algebra modules
that satisfy the corresponding property and IAC algebras. Indeed,
the authors of~\cite{eqqlie}  considered the exterior squares  of
some irreducible representations of the Lie algebra
$\mathfrak{sl}(2)$ in the same way as we consider
such representations for $\GL(2,\F)$ in Theorem~\ref{T:CG}
and they defined non-associative anti-commutative algebras as we do
in Theorem~\ref{T:GammaL}. In addition, they explicitly computed the structure
constant tables for the resulting 3, 7, and 11-dimensional algebras, noting
that the 3-dimensional algebra is the simple Lie algebra $\mathfrak{sl}
(2)$ (as implied also by our Proposition~\ref{prop:51}), while
  the 7-dimensional algebra is the non-Lie simple Malcev algebra.
  This fact is further exploited in~\cite{BD} for the study of the polynomial
  identities of the Malcev algebra.

In Section~\ref{sec2}, we review some known properties of UCS $p$-groups
before we establish in Section~\ref{sec3} a duality
between the class of UCS $p$-groups of exponent $p^2$, and the class of
IAC algebras. In Section~\ref{sec4} we prove that IAC algebras are semisimple,
and that their subalgebras correspond to powerful subgroups of UCS groups.
We address the classification of simple IAC algebras of dimension at most~4
in~Section~\ref{sec5}, and an infinite class of examples of simple IAC
algebras, with widely varying dimensions, is given in Section~\ref{sec:rdim}.
We prove a general version, which is valid also in prime characteristic,
of the Clebsch-Gordan decomposition
for tensor squares, exterior squares and symmetric squares of $\GL(2,\F)$-modules
in Section~\ref{sec:CG}, and apply it in Section~\ref{Simple} 
to construct some simple IAC algebras related to representations of
$\textup{SL}(2,\F)$.

\section{UCS \texorpdfstring{$p$}{}-groups of exponent \texorpdfstring{$p^2$}{}}\label{sec2}

Let us start by recalling some standard notions of finite group theory.
If $G$ is a group and $k$ is an integer, then $G^k$ denotes the
subgroup of $G$ generated by $k$-th powers. To avoid confusion, the $k$-fold
direct power of $G$ is denoted by $G^{\times k}$.
The commutator subgroup $G'$ of $G$
is the subgroup generated by all commutators $[x,y]=x^{-1}y^{-1}xy$
with $x,\ y\in G$. The Frattini subgroup $\Phi(G)$ is the
intersection of all the maximal subgroups of $G$. It is well-known that if $G$
is a finite $p$-group, then $\Phi(G)=G'G^p$. The center $\Z(G)$ of $G$
is the subgroup of $G$ consisting of elements $g\in G$ such that
$gx=xg$ holds for all $x\in G$. 

A subgroup $N$ of  a group $G$ is called \emph{characteristic}
if it is invariant under each automorphism $\alpha\in\Aut(G)$.
For instance, the subgroups $G'$, $G^k$, $\Phi(G)$ are characteristic.
A group $G$ with
a {\bf u}nique non-trivial proper {\bf c}haracteristic {\bf s}ubgroup
is abbreviated a \emph{UCS group}. Finite UCS groups were
studied by Taunt~\cite{taunt} and later finite UCS $p$-groups were
explored by the first and the third authors in collaboration with
P.\ P.\ P\'alfy~\cite{ucs}.  The characteristic subgroups of
a finite UCS $p$-group $G$ are
$1$, $\Phi(G)$, and $G$, and consequently the exponent of $G$ is
either~$p$ or~$p^2$.

It is useful to review some properties of UCS groups.

\begin{lemma}\label{lemma1}
\textup{\cite{ucs}*{Lemma~3}.}
Suppose that $G$ is a finite
non-abelian UCS $p$-group and $1\lhd N\lhd G$ are the only
characteristic subgroups of $G$. Then the following hold:
\begin{enumerate}[{\rm (a)}]
  \item $G'=\Phi(G)=\Z(G)=N$ and $G^p=1$ or $G^p=N$;
  \item the groups $G/N$ and $N$ are elementary abelian $p$-groups.
\end{enumerate}
\end{lemma}

\begin{lemma}\label{lemma2}
\textup{\cite{ucs}*{Theorem~4}.}
Let $G$ be a finite $p$-group such that $G/\Phi(G)$ and $\Phi(G)$
are non-trivial
elementary abelian $p$-groups. Then $G$ is a UCS group if and only if $\Aut(G)$
induces an irreducible linear group on both $G/\Phi(G)$ and on $\Phi(G)$.
\end{lemma}

Consider, for a group $G$, the natural homomorphism $G\to G/\Phi(G)$.
For $g\in G$ and for $H\le G$ we define
$$
\overline{g}=g\Phi(G)\quad\mbox{and}\quad \overline{H}=H\Phi(G)/\Phi(G).
$$
In particular, 
$\overline{G}=G/\Phi(G)$.

\begin{lemma}\label{lemma3}
\textup{\cite{ucs}*{Theorem~4(iii)}.}
Let $G$ be a finite UCS $p$-group of odd exponent $p^2$. Then the map
$\varphi\colon \overline G\to\Phi(G)$ defined by $\overline{g}\varphi=g^p$
is a well-defined  isomorphism
between the $\F_p\Aut(G)$-modules $\overline G$ and $\Phi(G)$. In particular,
$|\overline G|=|\Phi(G)|$.
\end{lemma}

Suppose that $G$ is a UCS $p$-group of odd exponent $p^2$ and let $g\in G$.
By Lemma~\ref{lemma3}, defining $\overline g^p=\overline g\varphi=g^p$,
we obtain a well-defined
isomorphism between the $\F_p\Aut(G)$-modules $\overline G$ and $\Phi(G)$.
If $h\in\Phi(G)$, then
let $h^{1/p}$  denote the unique preimage $h\varphi^{-1}\in\overline G$ of $h$
under $\varphi$. Thus
for every $\overline{g}\in \overline{G}$ and $h\in \Phi(G)$ we have
\begin{equation}\label{E:1}
  (\overline g^p)^{1/p}=\overline g\quad\mbox{and}\quad
  (h^{1/p})^p=h.
\end{equation}

\begin{theorem}\label{dirprodUCS}
\textup{\cite{taunt}*{Theorem~2.1}.}
Suppose that $G$ is a UCS group, and $N$ is its proper 
non-trivial characteristic subgroup. Assume further that
$\Aut(G)$ fixes no non-trivial element of $N\cup G/N$. 
Then for all $k\geq 1$ the $k$-th direct power $G^{\times k}$ 
is a UCS group, and $N^{\times k}$ is a characteristic subgroup of $G^{\times k}$.
\end{theorem}
  
\section{UCS \texorpdfstring{$p$}{}-groups and IAC algebras}\label{sec3}

We consider non-associative\footnote{Precisely, we assume that the binary operation $L\times L\to L$ may or may not be associative.} algebras
$L$ over a field $\F$ satisfying $xx=0$ for all $x\in L$. Our
focus will be when $\Char(\F)\ne2$. Such an algebra is called
\emph{anti-commutative} because $yx=-xy$ holds for all $x,y\in L$.
As these algebras are more similar to Lie algebras than to associative algebras,
we henceforth write the product $xy$ as $\prd{x}{y}$.
An anti-commutative algebra $L$ for which $\Aut(L)$ acts 
irreducibly on $L$ will be called {\em IAC algebra} (for {\bf i}rreducible
and {\bf a}nti-{\bf c}ommutative).

Let $p$ be an odd prime and let $G$ be a UCS $p$-group of exponent $p^2$.
By Lemma~\ref{lemma3}, the map $\overline G\rightarrow \Phi(G)$,
defined by $\overline x\mapsto x^p$ is
an isomorphism of $\Aut(G)$-modules.
Given $\lambda\in\F_p$ and $\overline x\in\overline G$, considering
$\lambda$ as an integer between $0$ and $p-1$, the scalar action
$\lambda\overline x=\overline {x^\lambda}$ is well-defined and turns
$\overline G$ into a vector space over $\F_p$. It is less obvious
that the following product $\prd{\overline x}{\overline y}$ 
for $\overline x,\ \overline y\in\overline G$ is well-defined:
\begin{equation}\label{estrela}
\prd{\overline x}{\overline y}=[x,y]^{1/p}\qquad\textup{where $x,\ y\in G$
and $[x,y]=x^{-1}y^{-1}xy$.}
\end{equation}
Henceforth $\iac G=(\overline G,+,\prd\cdot\cdot)$ will denote
this $\F_p$-algebra.

\begin{lemma}\label{ucstoiac}
  Suppose that $p$ is an odd prime and
  $G$ is a finite UCS $p$-group of exponent~$p^2$.
  Then $\iac G$ defined above is an IAC algebra over $\F_p$.
\end{lemma}
\begin{proof}
  We first prove that the operation $\prd\cdot\cdot$ given
  by equation~\eqref{estrela} is well-defined.
  Let $v_1,\ v_2,\ w_1,\ w_2 \in G$ such that
  $\overline v_1=\overline v_2$ and $\overline w_1=\overline w_2$.
  Then $v_2=v_1z_v$ and $w_2=w_1z_w$ with $z_v,\ z_w\in\Phi(G)$. 
  As $z_v$ and $z_w$ are central, and $G$ has nilpotency class~2, we
  obtain that $[v_2,w_2]=[v_1,w_1]$, and so 
$[v_2,w_2]^{1/p}=[v_1,w_1]^{1/p}$.  Thus the value of $\prd\cdot\cdot$
  is independent of the choice of coset representatives, and
  $\prd\cdot\cdot$ is well-defined.

  We prove next that $\prd\cdot\cdot$ distributes over the addition
  in $\overline G$.  Take $u, v, w\in G$ and note that $[u,vw]=[u,v][u,w]$ holds since $G$ has nilpotency class $2$. Therefore
\[
  \prd {\overline u} {\overline v+\overline w}=[u,vw]^{1/p}= ([u,v][u,w])^{1/p}=\\{}[u,v]^{1/p}[u,w]^{1/p}=  \prd uv+\prd uw.
\]

In addition,  $\prd{\overline v}{\overline v}=[v,v]^{1/p}= 1^{1/p}$, 
equals the zero element of $\overline G$.
  Hence $\prd{\overline v}{\overline v}=0$ for all $\overline v\in\overline G$.
  An automorphism $\alpha\in\Aut(G)$ induces an 
  $\F_p$-linear transformation $\overline\alpha\colon \overline G\rightarrow \overline G$ and a (restricted) isomorphism $\alphadown\colon \Phi(G)\rightarrow\Phi(G)$. 
Then, by Lemma~\ref{lemma3}, the map $\varphi\colon\overline x\mapsto x^p$
intertwines $\alphadown$ and ${\overline\alpha}$, i.e.
\begin{equation}\label{phialphaintertw}
  \varphi(\alphadown)={\overline\alpha}\varphi\qquad\textup{or}\qquad
  (\alphadown)\varphi^{-1}=\varphi^{-1}{\overline\alpha}.
\end{equation}
Hence
\begin{align*}
  \prd{\overline{u}}{\overline{v}}\overline{\alpha}
  &=[u,v]^{1/p}\overline{\alpha}
  =[u,v]\varphi^{-1}\overline{\alpha}
  =[u,v](\alphadown)\varphi^{-1}
  =[u\alpha,v\alpha]\varphi^{-1}\\
  &=[u\alpha,v\alpha]^{1/p}
  =\prd{\overline{u}\,\overline{\alpha}}{\overline{v}\,\overline{\alpha}}.
\end{align*}
Thus $\overline{\alpha}\in\Aut(\iac G)$. By Lemma~\ref{lemma2},
$\Aut(G)$ induces an irreducible subgroup of $\GL(\overline G)$.
Thus $\Aut(\iac G)$  is irreducible on $\overline G$,
and $\iac G$ is an IAC algebra as claimed.
\end{proof}

In characteristic~2, we need not have $|G/\Phi(G)|=|\Phi(G)|$ as
in Lemma~\ref{lemma3}. For example, a non-abelian group $G$ of order 8 is
UCS and satisfies
$|G/\Phi(G)|=4>2=|\Phi(G)|$.
Although our construction of
$\iac G$ works more generally, namely when $p>2$ and $G$ satisfies
$1<\Phi(G)\le\Z(G)$, $|\overline G|=|\Phi(G)|$, and  $G^{p^2}=1$, the resulting anti-commutative algebra
$\iac G$ need not be an irreducible $\Aut(L)$-module when $G$ is not
a UCS group.

Given an algebra $L$, the subspace $\prd LL$ generated by all
products $\prd xy$ with $x,\ y\in L$ is invariant under
$\Aut(L)$. Thus if $L$ is an IAC algebra, we have
$\prd LL=0$ or $\prd LL=L$. If $L=\iac G$ for some finite UCS $p$-group $G$
of exponent $p^2$, then $\prd LL=0$ occurs if and only if $G$ is abelian.
We will usually assume that this is not the case; that is,
$\prd LL=L$ holds.
Adopting the terminology from Lie algebras, an anti-commutative
algebra $L$ is said to be {\em abelian} if $\prd LL=0$.

The following result shows that the construction in Lemma~\ref{ucstoiac}  
can be reversed.

\begin{theorem}\label{prop:iac2ucs} Given an odd prime $p$ and a finite-dimensional IAC algebra
  $L$ over~$\F_p$, 
  there exists a finite UCS  $p$-group $G=\mathcal{G}(L)$ of exponent $p^2$
  such that $\iac G\cong L$.
\end{theorem}
In Theorem~\ref{th:bij} we show that $\mathcal{G}=\mathcal{L}^{-1}$, see
Remark~\ref{R:iso}.
\begin{proof}
  Suppose that $L$ has (finite) dimension~$r$ over $\F_p$ and
  $\{v_1, \dots, v_r\}$ is a basis.
  If $L$ is abelian, then we take $G$ to be the homocyclic group
  $(C_{p^2})^r$. Clearly, $\iac G\cong L$ in this case, and so in the rest
  of the proof we assume that $L$ is a non-abelian IAC algebra.

  Let $H=H_{p,r}$ be the $r$-generator free group in the variety of groups that
  have exponent dividing $p^2$, nilpotency class at most
  2 and have the property that
  all $p$-th powers are central. 
  Assume that $c_{k}^{(i,j)}$ are the structure constants of $L$; that is, 
  \begin{equation}\label{strconsteq}
    \prd{v_i}{v_j}=\sum_{k=1}^{r}c_k^{(i, j)}v_k\end{equation}
    for all $1\leq i< j\leq r$.
    We consider the constants $c_k^{(i,j)}$ as elements of the field $\F_p$ as
    well as integers in $\{0,\ldots,p-1\}$.
		Suppose that
  $h_1, \dots, h_r$ are  generators of $H$.
    Define $N\leq H$ as the subgroup
    \begin{equation}\label{defN}
    N=\left\langle
    [h_i, h_j]^{-1}\prod_{k=1}^r(h_k^p)^{c_k^{(i,j)}}\mid 1\leq i< j\leq r
    \right\rangle.
    \end{equation}
    As $N\leq \Phi(H)\leq \Z(H)$, the subgroup
    $N$ is normal in $H$.
    Further, considering $\Phi(H)$ as an $\F_p$-vector space, the
    given generators of $N$ are linearly independent,
and hence $|N|=p^{\binom{r}{2}}$. 

    Set $G=H/N$. Then  $G$ is a finite $p$-group, and so
    the Frattini subgroup of $G$ is
    \[
    G^pG'=(H/N)^p(H/N)'=(H^pH'N)/N=\Phi(H)/N.
    \]
    On the other hand, since $\Phi(G)=G^pG'$ and the relations in $N$
    imply that $G'\leq G^p$,
    we find that $\Phi(G)=G^p$. Since $\Phi(G)\leq \Z(G)$
    and $\Phi(G)^p=1$, the map $\varphi\colon\overline{G}\to\Phi(G)$ defined by
    $\overline g\mapsto g^p$ is a surjective $\Aut(G)$-module homomorphism.
    Since $|\Phi(H)|=p^{r+\binom{r}{2}}$ by~\cite{ucs}*{p.\;87}, we have
    $|\overline G|=|\Phi(G)|=p^r$, and the map $\overline g\mapsto g^p$
    is an isomorphism.

    As remarked before this theorem, $G$ satisfies sufficient conditions
    for us to construct the anti-commutative algebra $\iac{G}$,
    as in Lemma~\ref{ucstoiac}. Set $g_i=h_iN$ for $1\leq i\leq r$, and
    define the linear map $\xi\colon L\rightarrow \overline G$
    by $v_i\xi=\overline g_i$ for all $i$. We claim that
    $\xi$ is an isomorphism of $\F_p$-algebras. Indeed, by~\eqref{estrela}
 and~\eqref{defN} we have
    \[
    \prd{v_i\xi}{v_j\xi}=\prd{\overline g_i}{\overline g_j}=[g_i, g_j]^{1/p}=
    \left(\prod_{k=1}^r{(g_k^p)^{c_k^{(i,j)}}}\right)^{1/p}.
    \]
   Using~\eqref{E:1} and~\eqref{strconsteq} this equals
\[
    \prod_{k=1}^r\left((\overline{g}_k^p)^{1/p}\right)^{c_k^{(i,j)}}
    =\prod_{k=1}^r{(\overline g_k)^{c_k^{(i,j)}}}=\prd{v_i}{v_j}\xi.
\]
Since $|L|=|\overline{G}|=p^r$ and $\xi$ is surjective, it follows that $\xi$ is an $\F_p$-isomorphism.

    We claim that $G$ is a UCS $p$-group of exponent $p^2$.
    The construction of the $p$-group $G$ ensures that $\overline{G}$ and $\Phi(G)$ are elementary abelian. By Lemma~\ref{lemma2},
    it suffices to show that $\Aut(G)$ acts irreducibly on
    $\overline G$ and on $\Phi(G)$.
    Fix $\alpha\in \Aut(L)$ and write
    \[
    v_i\alpha=\sum_{k=1}^r a_{ik} v_k\qquad\textup{where $i\in\{1,\dots,r\}$ and $a_{ik}\in\F_p$.}
    \]
As $H$ is a relatively free group, there is a homomorphism
    $\alpha^*\colon H\rightarrow H$ satisfying
    \[
    h_i\alpha^*=\prod_{k=1}^rh_k^{a_{ik}}\qquad\textup{where $a_{ik}\in\{0,1,\dots,p-1\}$}.
    \]
  By Burnside's Basis Theorem, $\alpha^*\in \Aut(H)$.
	
    We claim that $N$ is invariant under $\alpha^*$. Define $\psi\colon H'\rightarrow
    H^p$ as the linear map that acts on the generators of $H'$ by $[h_i,h_j]\psi=\prod_{k=1}^r(h_k^p)^{c_k^{(i,j)}}$ for  all  $i<j$.
    Note that the subgroup $N$ equals $\{h^{-1}(h\psi)\mid h\in H'\}$. It suffices to show that $\psi\alpha^*=\alpha^*\psi$, since this equality
    implies that
    $(h^{-1}(h\psi))\alpha^*=(h\alpha^*)^{-1}(h\alpha^*)\psi$ for all $h\in H'$, and, in turn, that $N$ is invariant under $\alpha^*$. 
    Define $\gamma\colon H^p\rightarrow L$  by $h_i^p\gamma=v_i$ and observe
    that $\gamma$ is a linear
     isomorphism. Similarly $\gamma\wedge\gamma\colon H'\rightarrow
    \wedge^2 L$ defined by $[h_i,h_j]\mapsto v_i\wedge v_j$
    is a linear isomorphism onto the exterior square $\wedge^2 L$ of $L$.
    The map $M\colon\wedge^2L\to L$ defined by $(u\wedge v)M=\prd uv$ is
    an epimorphism since $L=\prd{L}{L}$.
    These maps are illustrated in Diagram~\eqref{diag2}:
    \begin{equation}\label{diag2}
    \begin{tikzcd}
      H' \arrow[r,"\gamma\wedge\gamma"]\arrow[d,"\alpha^*"']
      \arrow[rrr,bend left, "\psi"]& \wedge^2 L
      \arrow[r,"M"]\arrow[d,"\alpha\wedge\alpha"] & L
      \arrow[r,"\gamma^{-1}"]\arrow[d,"\alpha"] & H^p\arrow[d,"\alpha^*"]\\
      H' \arrow[r,"\gamma\wedge\gamma"']\arrow[rrr,bend right,"\psi"'] & \wedge^2 L \arrow[r,"M"'] & L
      \arrow[r,"\gamma^{-1}"'] & H^p.
    \end{tikzcd}
    \end{equation}
Since each square piece of Diagram~\eqref{diag2} is 
commutative, it follows that $\psi=(\gamma\wedge\gamma)M\gamma^{-1}$.  Therefore $\psi\alpha^*=\alpha^*\psi$,
as claimed, which, in turn, implies that $N$ is invariant under $\alpha^*$.
Hence $\alpha^*$ induces an automorphism $\widetilde\alpha$ of $G=H/N$.
Recall that $\xi\colon L\to\overline{G}$ satisfies $v_i\xi=\overline{g}_i$.
It is straightforward to see that the following Diagram~\eqref{D:2} is commutative:
\begin{equation}\label{D:2}
\begin{tikzcd}
      H \arrow[r, twoheadrightarrow ]\arrow[d,"\scalebox{1.3}{$\alpha^*$}"']
      & H/N=G
      \arrow[r,twoheadrightarrow,"\overline{\phantom{N}}"]\arrow[d,"\scalebox{1.3}{$\widetilde{\alpha}$}"] & G/\Phi(G)=\overline{G}
      \arrow[r,"\scalebox{1.3}{$\xi^{-1}$}"]\arrow[d,"\scalebox{1.3}{$\overline{\alpha}$}"'] & L\arrow[d,"\scalebox{1.3}{$\alpha$}"]\\
      H \arrow[r,twoheadrightarrow] & H/N=G \arrow[r,twoheadrightarrow,"\underline{\phantom{N}}"'] & G/\Phi(G)=\overline{G}
      & L \arrow[l,"\scalebox{1.3}{$\xi$}"].
\end{tikzcd}
\end{equation}
Hence the equation $\overline{\alpha}=\xi^{-1}\alpha\xi$ relates $\alpha\in\Aut(L)$
and $\overline{\alpha}\in\Aut(\overline{G})$.
Therefore $\xi^{-1}\Aut(L)\xi$ is contained
in the group induced by $\Aut(G)$ on $\overline G$, and so
$\Aut(G)$ acts irreducibly on $\overline G$. At the start of the proof we
showed that $\overline G$ and $\Phi(G)$ are isomorphic as $\Aut(G)$-modules.
Thus $\Aut(G)$ is irreducible on $\Phi(G)$ also, and
by Lemma~\ref{lemma2}, $G$ is a UCS $p$-group of exponent $p^2$.
\end{proof}

Finite $p$-groups are commonly expressed using
\emph{polycyclic presentations}, see~\cite{handbookcgt}*{Chapter~8}.

\begin{corollary}\label{C:PCP}
Let $L$ be an IAC algebra over $\F_p$, $p$ odd, with basis $\{v_1,\dots,v_r\}$
and suppose $\prd{v_i}{v_j}=\sum_{k=1}^{r}c_k^{(i, j)}v_k$ where 
$c_k^{(i,j)}\in\F_p$. Then the
UCS group $\mathcal{G}(L)$ defined in Theorem~\textup{\ref{prop:iac2ucs}}
is isomorphic to the group given by the polycyclic presentation
\begin{multline}\label{E:PCP}
  G=\langle g_1,\dots,g_r,z_1,\dots,z_r\mid \\g_i^p=z_i,\hskip2mm z_i^p=[z_j,g_i]=[z_i,z_j]=1,\hskip2mm
  [g_i,g_j]=\prod_{k=1}^{r}z_k^{c_k^{(i, j)}}, i<j\rangle.
\end{multline}
\end{corollary}

\begin{proof}
Identify the scalars $c_k^{(i, j)}\in\F_p$ in~\eqref{E:PCP} with
integers in the set $\{0,1,\dots,p-1\}$. The group presentation~\eqref{E:PCP}
is called a polycyclic presentation, and it has the property that
the elements of the group defined by the presentation
can be expressed in the form
$g_1^{x_1}\cdots g_r^{x_r}z_1^{x_{r+1}}\cdots z_r^{x_{2r}}$ where the exponents
satisfy $0\le x_i\le p-1$
for $1\le i\le 2r$. This shows that $|G|\le p^{2r}$. If equality holds, the
presentation is called \emph{consistent}, see~\cite{handbookcgt}*{page~280}.

Recall in the proof of Theorem~\textup{\ref{prop:iac2ucs}} that
$\mathcal{G}(L):=H/N$ where $H=H_{p,r}$ and $N$ is the normal subgroup
defined by~\eqref{defN}. Suppose, as in the proof of Theorem~\textup{\ref{prop:iac2ucs}}, that $H$ is generated by $h_1,\ldots,h_r$.  Since the map $\chi\colon G\to H/N$ defined
by $g_i\chi=h_iN$ and $z_i\chi=h_i^pN$ preserves the relations in~\eqref{E:PCP},
it is a homomorphism. However $\chi$ is surjective, so $\chi$ is an isomorphism, and the
presentation~\eqref{E:PCP} must be consistent.
\end{proof}

As introduced in the introduction, for a prime $p$, we let $\IAC_p$ denote the set of isomorphism classes of finite dimensional IAC algebras over~$\F_p$ and we let $\UCS_{p^2}$ denote the set of isomorphism classes of finite
UCS $p$-groups of exponent~$p^2$.

\begin{remark}\label{R:iso}
  Let $p$ be an odd prime.
  We have well-defined maps on isomorphism classes
  $\UCS_{p^2}\rightarrow \IAC_p$ with $[G]\mapsto[\mathcal{L}(G)]$ as per
  Lemma~\ref{ucstoiac} and
  $\IAC_p\rightarrow \UCS_{p^2}$ with $[L]\mapsto[\mathcal{G}(L)]$ as per
  Theorem~\ref{prop:iac2ucs}. It is convenient to identify isomorphism
  classes $[G]$ and $[L]$ with $G$ and $L$ respectively, and use $=$ instead of $\cong$.
  With this abuse in mind, $\mathcal{L}$ and $\mathcal{G}$ can be viewed
  as mutually inverse functions.
\end{remark}

\begin{theorem}\label{th:bij}
  Let $p$ be an odd prime.
  \begin{enumerate}[{\rm (a)}]
\item If $G$ is a finite UCS group of exponent $p^2$, then $\G(\iac G)= G$.
  \item If $L$ is a finite-dimensional IAC algebra over $\F_p$,
  then $\iac{\G(L)}= L$.
  \end{enumerate}
  Therefore the maps $\G\colon  \IAC_p\rightarrow \UCS_{p^2}$ and 
  $\mathcal{L}\colon \UCS_{p^2}\rightarrow \IAC_p$ are bijections.
\end{theorem}

\begin{proof}
(a)~Let $G$ be a finite UCS $p$-group of exponent $p^2$, say with~$r$
generators. Suppose that
$\{g_1,\ldots,g_r\}$ is a minimal generating set for $G$.
Since $G'=G^p$ there exist constants $c_k^{(i,j)}$ satisfying:
\begin{equation}\label{E:comm}
  [g_i,g_j]=\prod_{k=1}^r(g_k^p)^{c_k^{(i,j)}}\mbox{ where $1\leq i<j\leq r$ and $0\leq c_k^{(i,j)}\leq p-1$.}
\end{equation}
Introducing redundant generators $z_k:=g_k^p$ we see that $G$ has a polycyclic
presentation of the form~\eqref{E:PCP}.
It follows from ~\eqref{estrela} that $\prd{\overline{g}_i}{\overline{g}_j}=
\sum_{k=1}^r c_k^{(i,j)}\overline{g}_k$. Therefore the numbers $c_k^{(i,j)}$, interpreted as elements of the field $\F_p$, are
structure constants of $L:=\iac G$. By Corollary~\ref{C:PCP}
the group $\grp{L}$ has a polycyclic presentation of the form~\eqref{E:PCP}
involving the \emph{same} $c_k^{(i,j)}$. This proves
that $\grp{\iac{G}}\cong G$ as desired.

(b)~Let $L$ be a finite-dimensional IAC algebra over $\F_p$.
Theorem~\ref{prop:iac2ucs} proves precisely that $\iac{\grp{L}}\cong L$ where
the minimal number of generators of $\grp{L}$ is $\dim(L)$.
\end{proof}

We denote the kernel of the homomorphism $\Aut(G)\to\Aut(G/\Z(G))$
by $\Aut_C(G)$. The elements of $\Aut_C(G)$ are
called {\em central automorphisms} of $G$.

\begin{theorem}\label{th:aut}
Let $G$ be an $r$-generator UCS $p$-group of exponent $p^2$ where $p>2$, and let $L=\iac{G}$.
Then $\Aut(G)/\Aut_C(G)\cong\Aut(L)$ where $\Aut_C(G)$ is elementary abelian of
order $p^{r^2}$.
\end{theorem}

\begin{proof}
Let $G$ be an $r$-generator UCS $p$-group of exponent $p^2$. Write $A=\Aut(G)$.
Each $\alpha\in A$ induces and $\overline{\alpha}\in\Aut(\overline{G})$ defined
by $\overline{g}\overline{\alpha}=(g\alpha)\Phi(G)$ and
$\alpha\mapsto\overline{\alpha}$ is a homomorphism. The additive group
of $L:=\iac{G}$ is $\overline{G}$, and it was shown in the proof of
Theorem~\ref{prop:iac2ucs} that $\alpha\mapsto\overline{\alpha}$ defines a
\emph{surjective} homomorphism $\Aut(G)\to\Aut(L)$. The kernel $K$ of this
epimorphism comprises those $\alpha\in A$ with $\overline g{\,}\overline \alpha=\overline g$ for all $g\in G$. Since $\Phi(G)=\Z(G)$,
$K$ is  the group $\Aut_C(G)$ of central automorphisms.
Therefore $\Aut(G)/\Aut_C(G)\cong\Aut(L)$ as claimed.

Let $\alpha\in K$ be a central automorphism. Then $g \alpha=gz_g$ for
some $z_g\in Z(G)$ for all $g\in G$.
Since
$[g_1\alpha,g_2\alpha]=[g_1,g_2]$ and $G'=\Z(G)$, we see that $\alpha$
acts as the identity on $\Z(G)$. Hence
$z_g$ depends only on the coset $g\Phi(G)=\overline g$ and we may
write $z_{\overline g}=z_g$.  This gives a function
$\overline{G}\to\Z(G)$ defined by  $\overline{g}\mapsto z_{\overline{g}}$.
There is a well-known isomorphism
$\Aut_C(G)\to\textup{Hom}_{\F_p}(\overline{G},\Z(G))$ taking $\alpha$
to the map $\overline{g}\mapsto z_{\overline{g}}$. However,
$\textup{Hom}_{\F_p}(\overline{G},\Z(G))$ is isomorphic to the additive group
of $r\times r$ matrices over $\F_p$. Thus $\Aut_C(G)$ is elementary abelian of
order $p^{r^2}$. (By contrast, the group $\textup{Inn}(G)$ of
inner automorphisms, which is a subgroup of $\Aut_C(G)$, has order only~$p^r$.)
\end{proof}

\section{The structure of UCS \texorpdfstring{$p$}{}-groups and IAC algebras}\label{sec4}

In this section we prove that a finite-dimensional IAC algebra
over an arbitrary field
is semisimple in the sense that it is a direct sum of
pairwise isomorphic simple algebras.
This leads naturally to the study of simple
IAC algebras in Section~\ref{sec5}.
As noted in the introduction, invariant non-associative algebra
structures also appeared in the study of finite simple groups, and
a result similar to our Theorem~\ref{SimpleSum} was proved in~\cite{kmr}*{Lemma~2.4}. Since our context is
somewhat different, we present our theorem with a proof. 
We will also extend the
group--algebra duality established in Section~\ref{sec3}.
In Proposition~\ref{pembedded} we exhibit a bijection between subalgebras
of $L$ and powerful subgroups of $G=\grp{L}$ that contain $\Phi(G)$.
This is reminiscent of the duality between field extensions and groups in
Galois theory, because the substructures match.

Let $V$ be a vector space. An irreducible subgroup $H$ of $\GL(V)$ is called
\emph{imprimitive} if there is an $H$-invariant
direct sum decomposition
$V=V_1\oplus\dots\oplus V_\ell$ with $\ell\geq 2$,
We call the decomposition $V=V_1\oplus \dots \oplus V_\ell$ an {\em imprimitivity decomposition} for $H$. If no such
decomposition exists, we call $H$ \emph{primitive}.
By the irreducibility of the acting group, an imprimitivity decomposition is
necessarily equidimensional; that is $\dim V_i=\dim V_j$ for all subspaces
$V_i,\ V_j$ of the decomposition.

Suppose that $V=V_1\oplus\dots\oplus V_\ell$ is an imprimitivity
decomposition for an irreducible subgroup $H\leq \GL(V)$.
Then the action of $H$ induces a permutation group on the set
$\{V_1,\ldots,V_\ell\}$. Further, $H$ is isomorphic (as a linear group) to
a subgroup of
the wreath product $\GL(r/\ell,\F)\wr K$ where  $r=\dim V$ and 
$K\leq\Sym_\ell$ is permutationally isomorphic to the group induced by $H$ on $\{V_1,\ldots,V_\ell\}$.
The structure of $\GL(r/\ell,\F)\wr\Sym_\ell$ is described
in~\cite{suprunenko}*{Section~15}.
The following theorem is commonly attributed to Clifford,
see~\cite{suprunenko}*{Lemma~4, Section~15}.

\begin{theorem}\label{th:wreath}
The following hold.
\begin{enumerate}[{\rm (a)}]
\item  Given a vector space $V_1$
  over a field $\F$, a non-trivial subgroup $A_1\le \GL(V_1)$, and a transitive
subgroup $P\le\Sym_\ell$, the subgroup $A_1\wr P$ of $\GL((V_1)^{\oplus\ell})$
is irreducible if and only if $A_1$ is irreducible on~$V_1$.
\item Suppose that $A\leq \GL(V)$ is irreducible on the vector space $V$
and preserves the imprimitivity
decomposition $V=V_1\oplus\cdots\oplus V_\ell$. Then $A$ is conjugate in $\GL(V)$ to a subgroup of $\GL(V_1)\wr\Sym_\ell$. Further, $A$ is transitive on
$\{V_1,\dots,V_\ell\}$ and the setwise stabilizer $A_{V_1}$ is irreducible on $V_1$.
\end{enumerate}
\end{theorem}

Using Lie theoretic notation, the {\em center} of an
anti-commutative algebra $L$ is
\[
  \Z(L):=\{x\in L\mid \prd xy=0\mbox{ for all }y\in L\}.
\]
Since $\Z(L)$ is invariant under $\Aut(L)$, we see,
for a non-abelian IAC algebra $L$,  that $\Z(L)=0$.

The next theorem shows that an IAC algebra (over an arbitrary field)
is semisimple. See~\cite{kmr}*{Lemma~2.4} for a similar result.

\begin{theorem}\label{SimpleSum}
  The following hold.
  \begin{enumerate}[{\rm (a)}]
  \item Suppose $L_0$ is a simple IAC algebra, and $\ell\ge1$ is an integer.
    Then the direct sum $L=(L_0)^{\oplus\ell}$ is an IAC algebra.
    Further, if $L_0$ is non-abelian, then
    $L=(L_0)^{\oplus\ell}$ has precisely $\ell$ minimal ideals namely 
    the given direct summands, and in particular $\Aut(L)=\Aut(L_0)\wr \Sym_\ell$.
  \item Let $L$ be an IAC algebra of finite dimension. Then
  $L\cong I^{\oplus\ell}$ for some simple IAC subalgebra $I$ of $L$ and some $\ell\ge1$.
  Further, if $L$ is non-abelian, then $\ell$ is the number of minimal ideals of $L$.
  \end{enumerate}
\end{theorem}

\begin{proof}
(a)~If $L_0$ is abelian, then so is $L$. Thus $\dim(L_0)=1$ and
  $\Aut(L)=\GL(L)$ acts irreducibly. Hence $L$ is IAC.

Suppose now that
$L_0$ is non-abelian.
Let $L_1,\ldots,L_\ell$ be the summands of the direct sum
decomposition of $L$. Then, for $i\geq 1$, $L_i$ is a minimal ideal of
$L$. We claim that each minimal ideal $I$ coincides with $L_i$ for
some $i\geq 1$.  Assume, seeking a contradiction, that $I\neq L_i$ for
all $i\geq 1$.  Then, the minimality of $I$ implies
$\prd{L_i}I\subseteq L_i\cap I=0$ for all $i\geq 1$. Thus
$\prd{L}I=0$, and so $I\leq \Z(L)$; this is a contradiction as
$\Z(L)=0$.  Therefore $L_1,\ldots,L_\ell$ are all the minimal ideals
of $L$ as claimed. Now $\Aut(L)$ contains the wreath product
$W=\Aut(L_0)\wr \Sym_\ell$ and $W$ is an irreducible subgroup of
$\GL(L)$ by Theorem~\ref{th:wreath}(a). Thus $L$ is an IAC algebra.
The group  $\Aut(L)$ permutes the minimal
ideals of $L$, and hence $L=L_1\oplus\cdots\oplus L_\ell$ is an imprimitivity
decomposition of $\Aut(L)$.  Further, the stabilizer of $L_i$ in $\Aut(L)$
induces a group of automorphisms of $L_i$, which shows,
by Theorem~\ref{th:wreath}(b), that $\Aut(L)\leq\Aut(L_0)\wr \Sym_\ell$.
Therefore $\Aut(L)=\Aut(L_0)\wr
\Sym_\ell$ as claimed.

(b) Let  $L$ be a finite-dimensional IAC algebra and set $A=\Aut(L)$. If $L$
is abelian, then $L=\F^{\oplus\ell}$ where $\ell=\dim(L)$,
and each copy of $\F$ is a simple IAC algebra of dimension~1.
Thus in this case the assertion is valid.
Suppose that $L$ is non-abelian. As noted above, $\Z(L)=0$.
Let $I$ be a minimal ideal of $L$ and  set $\mathcal I=\{I\alpha\mid \alpha\in A \}$. Define $K=\sum_{J\in\mathcal I}J$; then $K$ is a nontrivial ideal
of $L$ which is invariant under $A$. Thus $K=L$ because $A$ acts irreducibly
on $L$. We show that the sum $\sum_{J\in\mathcal I}J$ is direct. Given $J_0\in\mathcal I$, set
$\overline J_0=\sum_{J\in\mathcal I\setminus\{J_0\}} J$.
We are required to show that $J_0\cap\overline J_0=0$. If
$J\in\mathcal I\setminus\{J_0\}$, then $J_0\cap J\lhdeq L$, and, by
the minimality of $J$, we have that $J_0\cap J=0$. On the
other hand, $\prd{J_0}J\leq J_0\cap J$, and hence $\prd{J_0}J=0$.
Thus $J_0$ centralizes each of the elements of $\mathcal I\setminus\{J_0\}$,
which gives that $\prd{J_0}{\overline J_0}=0$.
Then
\[
  \prd{J_0\cap \overline J_0}L=\prd{J_0\cap \overline J_0}{J_0+\overline J_0}=
  \prd{J_0\cap \overline J_0}{J_0}=0,
\]
which implies that $J_0\cap \overline J_0\leq \Z(L)$, and, in turn, that
$J_0\cap \overline J_0=0$. Thus
$L$ is equal to the direct sum $\bigoplus_{J\in\mathcal I}J$ as claimed.
Write $L\cong I^{\oplus\ell}$ where $I$ is a minimal ideal of $L$ and
$\ell=|\mathcal I|$.

It remains to show that $I$ is IAC. As $I$ inherits anti-commutativity from $L$,
we need to show that $\Aut(I)$ is irreducible on $I$.
Since $\Aut(L)$ permutes the elements of $\mathcal I$, the
decomposition $L=\bigoplus_{J\in \mathcal I}J$ is an imprimitivity decomposition for
$\Aut(L)$. Let $A_I$ denote the setwise stabilizer of $I$ in $\Aut(L)$. Then
$A_I$ induces a group of automorphisms of $I$ that is irreducible on
$I$ by Theorem~\ref{th:wreath}(b). Hence $I$ is an IAC algebra.
\end{proof}

\begin{theorem}\label{SumPower}
Let $G$ be a finite UCS $p$-group of exponent $p^2$ for some odd prime $p$ and
let $k\geq 1$.
Then $G^{\times k}$ is a UCS $p$-group   of exponent $p^2$ and $\iac{G^{\times k}}\cong
\iac G^{\oplus k}$.
\end{theorem}

\begin{proof}
  The assertion of the theorem holds if $G$ is abelian, and so we
  may assume that $G$ is non-abelian.
  Certainly $G^{\times k}$ has exponent~$p^2$. Moreover,
  \[
    \Phi(G^{\times k})=\Phi(G)^{\times k}\quad\mbox{and}\quad
    G^{\times k}/\Phi(G^{\times k})\cong(G/\Phi(G))^{\times k}=\overline{G}^{\times k}.
  \]
By Lemma~\ref{lemma2},
$\Aut(G)$ is irreducible on $\overline{G}$ and on $\Phi(G)$.
Further, since $G$ is non-abelian,
$\overline{G}$ and $\Phi(G)$ are non-trivial $\Aut(G)$-modules.
Since $\Aut(G)\wr\Sym_k\le\Aut(G^{\times k})$,
it follows from Theorem~\ref{th:wreath}(a)
that $\Aut(G^{\times k})$ acts irreducibly on $G^{\times k}/\Phi(G^{\times k})$ and on
$\Phi(G^{\times k})$, and thus $G^{\times k}$ is
a UCS $p$-group, by Lemma~\ref{lemma2}.

We now prove that $\iac{G^{\times k}}\cong\iac G^{\oplus k}$.
Let $G_i$ denote the $i$-th copy of $G$ in the direct
product $G^{\times k}$. Set $L=\iac{G^{\times k}}$ and, for $i=1,\ldots,k$,
let $L_i=G_i\Phi(G^{\times k})/\Phi(G^{\times k})$. We claim that $L_i\lhdeq L$.
First, if $\overline u,\ \overline v\in L_i$ with $u,\ v\in G_i$, then
$\prd{\overline u}{\overline v}=[u,v]^{1/p}\in L_i$. Hence $L_i\leq L$.
If $g_i\in G_i$ and $g_j\in G_j$ with $i\neq j$, then $[g_i,g_j]=1$, and
so $\prd{\overline g_i}{\overline g_j}=0$. This shows that $L_i$ is an ideal
of $L$ for each $i\in\{1,\dots,k\}$. Since
\[
\iac{G^{\times k}}=G^{\times k}/\Phi(G^{\times k})=\bigoplus_{i=1}^kG_i\Phi(G^{\times k})/\Phi(G^{\times k})=\bigoplus_{i=1}^k L_i,
\]
and each $L_i$ is isomorphic to $\iac G$, we obtain that
$\iac{G^{\times k}}=\iac G^{\oplus k}$.
\end{proof}

Powerful $p$-groups were introduced by Lubotzky and Mann~\cite{powerful}.
A finite $p$-group $G$ is said to be {\em powerful} if $G'\leq G^p$ when $p>2$, and $G'\leq G^4$ when $p=2$. A subgroup $N$ of a $p$-group $G$
is said to be {\em powerfully embedded} in $G$ if $[N,G]\leq N^p$ when $p>2$
 and $[N,G]\leq N^4$ when $p=2$. A powerfully embedded subgroup of $G$
is always normal in $G$.

\begin{proposition}\label{pembedded}
Let $G$ be a finite UCS $p$-group of odd exponent $p^2$, and $L=\mathcal{L}(G)$.
\begin{enumerate}[{\rm (a)}]
 \item There is a bijection between the set of subalgebras of $L$ and the
  set of powerful subgroups of $G$ that contain $\Phi(G)$.
\item The ideals of $L$ correspond, via part \textup{(a)}, to
  powerfully embedded subgroups of~$G$.
\end{enumerate}
\end{proposition}

\begin{proof}
There is a bijection $H\leftrightarrow\overline{H}:=H/\Phi(G)$ between the
set of subgroups of $G$ that contain $\Phi(G)$, and the set of subgroups of
the quotient $\overline{G}=G/\Phi(G)$.
The subgroups of $\overline{G}$ are precisely
the linear subspaces of $L$. We show first that $H$ is a powerful subgroup
of $G$ if and only if $\overline H$ is a subalgebra of~$L$ which we henceforth
denote as $\overline H\leq L$.

Suppose that $\overline H\leq L$ and let $h_1,\ h_2\in H$.
Then $\prd{\overline{h_1}}{\overline{h_2}}\in \overline H$, and so
$\prd{\overline{h_1}}{\overline{h_2}}=\overline{h}$ with $h\in H$.
The definition in \eqref{estrela} of the product in $L$ implies that $[h_1,h_2]=h^p$,
which shows that $H'\leq H^p$, and hence $H$ is powerful.
Conversely, suppose that $H$ is powerful.
Suppose $h_1,\ h_2\in H$. Then $(h_1h_2)^p=h_1^ph_2^p$ since $p>2$,
so the subgroup
$H^p=\langle h^p\mid h\in H\rangle$ equals the subset $\{h^p\mid h\in H\}$.
Since $H'\leq H^p$, it follows that $[h_1,h_2]=h^p$ for some $h\in H$.
Thus $\prd{\overline{h_1}}{\overline{h_2}}=[h_1,h_2]^{1/p}=\overline{h}\in\overline{H}$
by~\eqref{estrela}, and so $\overline{H}\leq L$ .

Suppose now that $H$ is powerfully embedded in $G$. 
We have, for every $h\in H$ and $g\in G$,
that $[h,g]=h_0^p$ for some $h_0\in H$, and so
$\prd{\overline{h}}{\overline{g}}=[h,g]^{1/p}=\overline{h_0}\in\overline{H}$.
Hence $\overline{H}\lhdeq L$.
On the other hand, suppose $I\lhdeq L$. Then $I=\overline{H}$ for some $H\le G$
with $\Phi(G)\le H$. Since $\overline{H}$ is an ideal, for every
$h\in H$ and $g\in G$ we have
$\prd{\overline{h}}{\overline{g}}=\overline{h_0}$ for some $h_0\in H$.
Hence, $[h,g]^{1/p}=\overline{h_0}$, and so $[h,g]=h_0^p$. Thus $[H,G]\le H^p$
and $H$ is powerfully embedded in $G$. 
\end{proof}

By Theorem~\ref{SimpleSum}, every non-trivial proper
ideal of an IAC algebra is a direct summand.
Now we state and prove a corresponding result for UCS $p$-groups of exponent
$p^2$.

\begin{corollary}\label{cor:minH0}
  Let $G$ be a finite UCS $p$-group with odd exponent $p^2$. There exists
  a subgroup $H$ of $G$ satisfying the following conditions:
\begin{enumerate}[{\rm (a)}]
  \item $H$ is powerfully embedded in $G$ and $\Phi(G)< H$, and
  \item $H$ is minimal among the subgroups of $G$ satisfying \textup{(a)}.
\end{enumerate}
There exists a UCS subgroup $H_0$ of $G$ of exponent $p^2$ such that $G=(H_0)^k$ for some $k\geq 1$, and $H_0\Phi(G)= H$.
\end{corollary}

\begin{proof}
A subgroup~$H$ satisfying conditions (a) and (b) does exist because
$G'\le G^p$ and $G$ satisfies condition (a).
Set $L=\iac G$ and $L_0=H/\Phi(G)$. 
By Lemma~\ref{ucstoiac},  $L$ is an IAC algebra, and $L_0\lhdeq L$
by Proposition~\ref{pembedded}.
Since $L_0$ is minimal, and hence simple, $L=(L_0)^{\oplus k}$ for some $k\ge 1$
by Theorem~\ref{SimpleSum}(b). However, $H_0=\grp{L_0}$ is a UCS group
of exponent $p^2$ by Theorem~\ref{prop:iac2ucs}, and so too is $(H_0)^{\times k}$
by Theorem~\ref{SumPower}. However,
$\iac{(H_0)^{\times k}}=\iac{H_0}^{\oplus k}=\iac{\grp{L_0}}^{\oplus k}=(L_0)^{\oplus k}=L$.
Therefore $\iac{(H_0)^{\times k}}=\iac{G}$. We obtain, by Theorem~\ref{th:bij}, that
$(H_0)^{\times k}=\grp{\iac{(H_0)^{\times k}}}=\grp{\iac{G}}=G$, and similarly $L_0=\iac{\grp{L_0}}=\iac{H_0}$. Finally, since $L_0=H_0/\Phi(H_0)=H/\Phi(G)$,
we have $H=H_0\Phi(G)$. 
\end{proof}

By Theorem~\ref{dirprodUCS} (also by Theorem~\ref{SumPower}),
if $G$ is a UCS $p$-group of exponent $p^2$, then $G^{\times k}$ is also
a UCS $p$-group of exponent $p^2$. Hence in the class
$\UCS_{p^2}$, the groups that cannot be decomposed non-trivially
as direct powers  can be viewed as basic building blocks.
The final result of this section gives
sufficient and necessary conditions for a UCS $p$-group of exponent $p^2$ to
be indecomposable as a direct power of a smaller such group. The corollary follows
easily from Theorem~\ref{SumPower} and Corollary~\ref{cor:minH0}.

\begin{corollary}
  The following are equivalent for a non-abelian
  finite UCS $p$-group of odd exponent
  $p^2$:
  \begin{enumerate}[{\rm (a)}]
  \item $G$ cannot be written as a direct product $H^{\times k}$ where $H$ is a
    UCS $p$-group and $k\geq 2$;
  \item $G$ does not contain a powerfully embedded proper subgroup $N$ such that
    $\Phi(G)< N$;
  \item $\iac{G}$ is a simple algebra.
    \end{enumerate}
  \end{corollary}

\section{Simple IAC algebras of dimension at most 4}\label{sec5}

For an anti-commutative algebra of dimension $r$ we have $\dim\prd LL\leq \binom{r}{2}$.
Thus an $r$-dimensional IAC algebra is abelian and simple if $r=1$, and abelian and non-simple
if $r=2$.  In this section we classify  simple IAC algebras of dimension~3
and~4 in Proposition~\ref{prop:51}
and Theorem~\ref{th:52}, respectively.

\subsection{3-dimensional IAC algebras}\label{sec51}
The proof of the following proposition uses ideas from the proof
of~\cite{ucs}*{Lemma~9}.

\begin{proposition}\label{prop:51}
If $L$ is a non-abelian $3$-dimensional
IAC algebra over a field $\F$ of characteristic different from $2$, then
$L$ is a simple Lie algebra and
$\Aut(L)\cong\textup{\textsf{SO}}(3,\F)$.
\end{proposition}

\begin{proof}
Suppose that $L$ is a non-abelian 3-dimensional IAC algebra. Let
$e=\{e_1, e_2,\linebreak e_3\}$ be a basis for $L$ and set $f_1=\prd{e_2}{e_3}$, $f_2=\prd{e_3}{e_1}$
and $f_3=\prd{e_1}{e_2}$. Since $\prd LL=L$, we have that
$f=\{f_1, f_2, f_3\}$ is also a basis for $L$.
Write $f_i=\sum_{j=1}^3 a_{ij}e_j$ where $a_{ij}\in\F$ and $i\in\{1,2,3\}$,
and consider the $3\times 3$ invertible matrix $A=(a_{ij})$. We claim that $L$
satisfies the Jacobi identity if and only if $A$ is symmetric.
Since $L$ is 3-dimensional, $L$ satisfies the Jacobi identity if and only if
the following equation holds:
\[
  0=\prd{e_1}{\prd{e_2}{e_3}}+\prd{e_2}{\prd{e_3}{e_1}}+\prd{e_3}{\prd{e_1}{e_2}}=
  \prd{e_1}{f_1}+\prd{e_2}{f_2}+\prd{e_3}{f_3}.
\]
Using the linear combinations for the $f_i$, we obtain the equivalent
equation:
\[
  a_{12}\prd{e_1}{e_2} + a_{13}\prd{e_1}{e_3}+
  a_{21}\prd{e_2}{e_1}+a_{23}\prd{e_2}{e_3}+a_{31}\prd{e_3}{e_1}
  +a_{32}\prd{e_3}{e_2}=0.
\]
Finally, the last equation is equivalent to the statement that $A$
is a symmetric matrix.

We now prove that $A$ is a symmetric matrix.
Let $g\in \Aut(L)$ and let $[g]_e$ and $[g]_f$ denote
the matrices that represent $g$ in the bases $e$ and $f$,
respectively.
We claim that $[g]_f=\det [g]_e([g]_e)^{-t}$ where $(\cdot)^{-t}$
denotes the inverse transpose operation. Let $[g]_e=(g_{ij})$ and
$[g]_f=(f_{ij})$. As $g\in \Aut(L)$, we have	
\begin{align*}
  f_1g=&\prd{e_2}{e_3} g=\prd{e_2g}{e_3g}
      =\prd{g_{21}e_1+g_{22}e_2+g_{23}e_3}{g_{31}e_1+g_{32}e_2+g_{33}e_3}\\
      =&(g_{22}g_{33} - g_{23}g_{32})\prd{e_2}{e_3}
       +(g_{23}g_{31} - g_{21}g_{33})\prd{e_3}{e_1}\\
       &+(g_{21}g_{32} - g_{22}g_{31})\prd{e_{1}}{e_2}.
\end{align*}
This gives the first of the following equations, the others follow by cyclic permutations:
\begin{align}
  f_1g&=(g_{22}g_{33} - g_{23}g_{32})f_1+(g_{23}g_{31} - g_{21}g_{33})f_2
       +(g_{21}g_{32} - g_{22}g_{31})f_3;\notag\\
  f_2g&=(g_{32}g_{13} - g_{33}g_{12})f_1+(g_{33}g_{11} - g_{31}g_{13})f_2
       +(g_{31}g_{12} - g_{32}g_{11})f_3;\label{E:coeff}\\
  f_3g&=(g_{12}g_{23} - g_{13}g_{22})f_1+(g_{13}g_{21} - g_{11}g_{23})f_2
       +(g_{11}g_{22} - g_{12}g_{21})f_3.\notag
\end{align}
By Cramer's Rule, the $3\times3$ system~\eqref{E:coeff} has coefficient matrix 
$\det([g]_e)([g]_e)^{-t}$.
Therefore $[g]_f=\det([g]_e)([g]_e)^{-t}$, as claimed.
For $v\in L$, let $[v]_e$ and $[v]_f$ denote vector
representations of $v$ in the bases~$e$ and $f$, respectively.
The matrix $A=(a_{ij})$ acts as a basis transformation matrix from the basis
$e$ to the basis $f$ and $[v]_f=[v]_eA$. Therefore $[g]_e=A[g]_fA^{-1}$ and so
$[g]_e=A\det([g]_e)([g]_e)^{-t}A^{-1}$, which implies that
$[g]_eA[g]_e^t=\det([g]_e)A$. This says that $g$ preserves the 
bilinear form $(u,v)\mapsto [u]_eA([v]_e)^t$ up to a scalar multiple.
In the next paragraph, we prove that $A^t=A$.
  
Consider the group $\grp{A}  := \{B\in\GL(3,\F)\mid BAB^t=\det(B)A\}$.
The  previous paragraph shows that
$g\in\Aut(L)$ implies $[g]_e\in \grp{A}$.
Let $J$ be the skew-symmetric matrix $A-A^t$, and define
$\grp{J}=\{B\in \GL(3,\F)\mid BJB^t=\det(B)J\}$ similarly. 
Clearly $\grp{A}$ is a subgroup of~$\grp{J}$.
Note that if $w\in \ker(J)$ and $B\in\grp{J}$, then
\[
  wBJ=wBJB^tB^{-t}=\det(B) wJB^{-t}=0.
  \]
Hence $\grp{J}$ fixes the subspace
\[
  \ker(J)=\{v\in \F^3 \mid vJ=0\}.
\]
Since $\Aut(L)$, considered as a subgroup of $\GL(3,\F)$,  is contained in $\grp{A}$ and
$\grp A\le\grp{J}$ and $\Aut(L)$ acts irreducibly on $L$,
we see that $\ker(J)$ must be $0$ or $L$.
Further, since the matrix $J=A-A^t$ is skew-symmetric, we obtain that
\[
  \det(J)=\det(J^t)=\det(-J)=(-1)^3\det(J),\quad\textup{and so}\quad 2\det(J)=0.
\]
Thus, since $\Char(\F)\neq 2$, $\det(J)=0$. Therefore $\ker(J)\ne  0$, and hence $\ker(J)=L$.
Thus $A-A^t=J=0$, and $A$ is symmetric, as claimed. As explained in the
first paragraphs of this proof, we obtain
that $L$ is a Lie algebra. Since $L$ is non-abelian and IAC, Theorem~\ref{SimpleSum}
implies that $L$ is the direct sum of pairwise isomorphic simple
IAC algebras. As $\dim L=3$ and an IAC algebra of dimension 1  is abelian
(see the paragraph before Proposition~\ref{prop:51}),
the algebra $L$  must be
simple.

Let $g\in\Aut(L)$. Since $[g]_e\le\grp{A}$ we see that $[g]_eA([g]_e)^t=\det(g)A$. However,
$A$ is an invertible $3\times3$ matrix, so taking determinants
gives $\det(g)^2\det(A)=\det(g)^3\det(A)$. Hence $\det(g)=1$, and
$g$ preserves the symmetric bilinear form defined by $A$.
Hence we have shown that
$\Aut(L)$ lies in the subgroup $\textsf{SO}(3,\F)$ of $\grp{A}$.
Suppose $g\in \textsf{SO}(3,\F)$. The steps in the second paragraph
of the proof are reversible, so it follows that $g\in\Aut(L)$.
Thus we conclude that $\Aut(L)\cong\textsf{SO}(3,\F)$, as desired.
\end{proof}

Proposition~\ref{prop:51} can be reversed: 
a 3-dimensional simple Lie algebra is IAC. This follows
from the fact that the automorphism group of a 3-dimensional
simple Lie algebra is isomorphic to the irreducible
group $\textsf{SO}(3,\F)$ stabilizing   its Killing form.
If $\F$ is a finite field or an algebraically closed field
of characteristic different from 2, then
the only 3-dimensional simple Lie algebra
over $\F$, up to isomorphism, is the classical
Lie algebra $\mathfrak{sl}(2,\F)$ which is IAC.
More generally, if $n\geq 2$, and $\F$ is a field of characteristic 0 or
of characteristic $p$ such that $p\nmid n$, then the group $\GL(n,\F)$
acts on the simple Lie algebra $\mathfrak{sl}(n,\F)$ by the adjoint
action $g\colon x\mapsto x^g=g^{-1}xg$ for all $x\in \mathfrak{sl}(n,\F)$
and $g\in \GL(n,\F)$. Since this action
is irreducible (see~\cite{BHRD}*{Lemma~5.4.10}), the simple Lie
algebra $\mathfrak{sl}(n,\F)$ is IAC.
It would be interesting to study whether all known finite-dimensional
simple Lie algebras are IAC, but this problem goes beyond the scope
of this paper. By~\cite{gerhard}*{Hauptsatz},
if $L$ is a finite-dimensional Chevalley type  simple
Lie algebra in characteristic different from~2, then
$L$ is IAC.

Interesting examples of IAC algebras can also be found among
Malcev algebras.
For a field $\F$ of characteristic different from $2$, the algebra
$\mathbb O(\F)$
of octonions  can be viewed as an anti-commutative algebra
under the multiplication $\prd ab=ab-ba$. If $\Char(\F)\neq 3$, then the algebra
$(\mathbb O(\F),+,\prd\cdot\cdot)$ 
 is not a Lie algebra, since $\mathbb O(\F)$ is not associative.
Further, the algebra $(\mathbb O(\F),+,\prd\cdot\cdot)$  can be written as
$\F\oplus M$ where $M$ is a 7-dimensional anti-commutative algebra and
the simple exceptional group $G_2(\F)$ acts irreducibly on $M$
(see~\cite{wilson}*{Section~4.3}). Hence $M$ is an IAC algebra.
By~\cite{kuzmin}*{Theorem~3.11}, non-Lie central simple Malcev algebras
over fields of characteristic different from 2 or 3
are 7-dimensional and can be defined similarly to the algebra $M$ above.
By~\cite{kuzmin}*{Theorem~3.11}, over a finite field of
characteristic different from 2, there is a unique isomorphism type of
central simple non-Lie Malcev algebras; namely, the algebra $M$ defined above.

\subsection{4-dimensional IAC algebras}\label{subsec:4dim}

In~\cite{ucs}*{Theorem~17} the authors
classified 4-gene\-rator UCS $p$-groups with
exponent $p^2$. The bijection in Theorem~\ref{th:bij} gives a
classification of 4-dimensional
simple IAC algebras
over $\F_p$ where $p$ is an odd prime. Using the ideas of the
proof of~\cite{ucs}*{Theorem~17}, this classification is  extended
in Theorem~\ref{th:52} to
all finite fields of characteristic different from~2.

The concept of ESQ-mo\-dules was defined in~\cite{ucs}. Let us recall the
definition.
 
\begin{definition}\label{def:esq}
For a group $H$ and a field $\F$, an  $\F H$-module $V$ is
called an \emph{ESQ-module} (for {\bf e}xterior {\bf s}elf-{\bf q}uotient)
if there exists an $\F H$-submodule $U$ of $\wedge^2 V$ such
that~$\wedge^2 V/U\cong V$.
 In this case, $H$ is called an \emph{ESQ subgroup}
 of $\GL(V)$.
\end{definition}

If $A$ is a non-abelian finite-dimensional IAC algebra over a field $\F$,
then the product map $\psi\colon \wedge^2A \rightarrow A$, defined
by $(u\wedge v)\psi=\prd uv$, is an epimorphism of $\Aut(A)$-modules.
Hence $\wedge^2 A/\ker\psi\cong A$, and so $A$ is an irreducible
ESQ $\F\Aut(A)$-module.
Conversely, let $V$ be an irreducible ESQ-module for a group $H$
over a field $\F$ and let $U\leq \wedge^2 V$
such that $\wedge^2V/U\cong V$. Assume
that $\psi\colon \wedge^2V\rightarrow V$ is an epimorphism with
kernel $U$. Then we can define an anti-commutative product
on $V$ by setting $\prd uv=(u\wedge v)\psi$ and we obtain that
$(V,+,\prd\cdot\cdot)$ is an IAC algebra whose automorphism group contains
$H$ as a subgroup.

\begin{theorem}\label{th:52}
  Let $\F_q$ be a finite  field of $q$ elements and assume that the
  characteristic of $\F_q$ is not $2$. Then the following are valid.
  \begin{enumerate}[{\rm (a)}]
  \item If $\Char(\F_q)=5$, then there exists no $4$-dimensional simple IAC algebra
    over $\F_q$.
  \item If $q\equiv\pm 1\pmod 5$, then 
    there exists, up to isomorphism,
    a unique $4$-dimensional
    simple IAC algebra over $\F_q$. Further, this algebra is isomorphic to
    the algebra given by the presentation
\begin{align*}
  \langle e_1,e_2,e_3,e_4\mid\hskip1.3mm
  &\prd{e_1}{e_2}=   e_1+  e_2  +5e_3  +3e_4,&
  &\prd{e_1}{e_3}= -4e_1 -4e_2  +0e_3  -2e_4,\\
  &\prd{e_1}{e_4}=  2e_1 +4e_2  -4e_3  -2e_4,&
  &\prd{e_2}{e_3}= -3e_1 -1e_2  +1e_3  +3e_4\\
  &\prd{e_2}{e_4}=  2e_1 +0e_2  +4e_3  +4e_4,&
  &\prd{e_3}{e_4}= -3e_1 -5e_2   -e_3   -e_4\rangle.
\end{align*}
  \item If $q\equiv\pm 2\pmod 5$, then, up to isomorphism,
    there are exactly two 
    $4$-dimensional simple IAC algebras over $\F_q$; one of these algebras is given
    by the presentation in statement~{\rm (b)}.
  \end{enumerate}
  Moreover, the automorphism group of the algebra presented in~{\rm (b)} is
  $\AGL(1,5)$, while the automorphism group of the second algebra in~{\rm (c)}
  is $C_5$. 
  \end{theorem}
\begin{proof}
  Since the proof follows the ideas in~\cite{ucs}*{Theorem~17}, we
  only give a sketch; the details can be filled in by the
  reader consulting~\cite{ucs}*{Theorem~17}.

  (a) It
follows from~\cite{ucs}*{Theorem~16} that no irreducible
ESQ subgroup of
  $\GL(4,q)$ exists if $5\mid q$, and hence over finite fields
  of characteristic 5, there are no 4-dimensional simple IAC algebras.
  This proves~(a).

  Suppose $q=p^k$ where $p\ne5$ is a prime and $k\ge1$. By quadratic reciprocity,
  $5\in(\F_q^\times)^2$ if and only if $q\equiv\pm1\pmod 5$. To see this
  note first that $q\equiv\pm2\pmod 5$ implies $p\equiv\pm2\pmod 5$
  and $k$ is odd, so $5\not\in(\F_p^\times)^2$ and $5\not\in(\F_q^\times)^2$.
  On the other hand, if $5\in(\F_q^\times)^2$, then either $k$ is even and
  $5\in(\F_{p^2}^\times)^2\le(\F_q^\times)^2$, or $k$ is odd and
  $p\equiv\pm1\pmod 5$ and so $5\in(\F_p^\times)^2\le(\F_q^\times)^2$.
  
  (b)~Let $L$ be a $4$-dimensional simple IAC algebra
  over $\F_q$ and set $A=\Aut(L)$. By the discussion preceding this theorem,
  $A$ is an irreducible ESQ-subgroup of $\GL(4,q)$.  We can argue, as in
  the proof of~\cite{ucs}*{Theorem~17}, that $A\leq\AGL(1,5)$ contains
  $C_5$ and that if $A=\AGL(1,5)$, then $L$ is isomorphic to the algebra
  presented in~(b). By~\cite{ucs}*{Theorem~16},
   $5\in(\F_q^\times)^2$ if and only if
  no proper subgroup of $\AGL(1,5)$ is irreducible. Hence in this case
  the algebra $L$ presented in statement~(b) is, up to isomorphism,
  the unique simple IAC algebra of dimension 4 over $\F_q$, and
  $\Aut(L)=\AGL(1,5)$.

  (c)~Assume now that $q\equiv\pm 2\pmod 5$. One possibility
  by~\cite{ucs}*{Theorem~17} is that a simple 4-dimensional IAC
  algebra~$L$ is given by the presentation~(b), and $\Aut(L)=\AGL(1,5)$.
  Suppose now that $A\neq \AGL(1,5)$.
  The cyclic subgroup $C_5$ of $\AGL(1,5)$ is contained in $A$, and it
  is irreducible since $q\equiv\pm 2\pmod 5$. Let $V$ be a 4-dimensional vector space over $\F_q$.
  Let $a\in \GL(4,q)$ be an element of order $5$, which is unique up to
  conjugacy. Then the $\left<a\right>$-module $\wedge^2V$ can be written as
  $\wedge^2V=W\oplus C$ where $W\cong V$ and $C$ is the 2-dimensional
  fixed-point space of $a$.
  An IAC algebra structure on  $V$ is determined by an
  epimorphism $\varphi\colon \wedge^2V\rightarrow V$.
  As in the proof of~\cite{ucs}*{Theorem~17},
  there is a bijection between the
  set of such linear epimorphisms and the set of subspaces $N$
  of $V\oplus\wedge^2V$ such that $\dim N=4$, $N\cap V=0$, and
  $N\cap \wedge^2V=C$. If fact, identifying $V$ and $W$ with the field
  $\F_{q^4}$, we may define, for each $w\in W\setminus\{0\}$
  such a subspace $N_w$ exactly as in~\cite{ucs}*{Theorem~17}. Now using
  the argument  in~\cite{ucs}*{Theorem~17}, the multiplicative group
  of $\F_{q^4}$ acts on the set of subspaces $N_w$, inducing isomorphisms between
  the associated IAC algebras, with 5 orbits corresponding
  to the cosets of the multiplicative subgroup $(\F_{q^4}^\times)^5$. Then
  $\mbox{Gal}(\F_{q^4}:\F_q)\cong C_4$ also acts on the set of these subspaces $N_w$,
  also inducing isomorphism between the associated algebras, fusing
  the 4 orbits that correspond to non-trivial cosets and leaving the fifth
  orbit invariant. Therefore
  there are two $\F_{q^4}^\times\rtimes\mbox{Gal}(\F_{q^4}:\F_q)$-orbits
  on the set of subspaces $N_w$, and
  there are two isomorphism types of simple $4$-dimensional IAC  algebras.
\end{proof}

The IAC algebras given in Theorem~\ref{th:52}(b,c) are
members of an infinite family: take $t=5$ in the following theorem and note
that $\mbox{\rm A$\Gamma$L}(1,5)=\mbox{\rm AGL}(1,5)$.

\begin{theorem}
  If $t$ and $q$ are powers of distinct primes such that $t>3$, then
  there exists a $(t-1)$-dimensional IAC algebra over $\F_q$ whose automorphism
  group contains $\mbox{\rm A$\Gamma$L}(1,t)$ acting in its irreducible representation 
of dimension
  $t-1$.
\end{theorem}

\begin{proof}
  By~\cite[Theorem~18]{ucs} the $(t-1)$-dimensional module for
  $\mbox{\rm A$\Gamma$L}(1,t)$ over $\F_q$ is irreducible and ESQ.
  The result now follows from the preamble to Theorem~\ref{th:52}.
\end{proof}

\section{\texorpdfstring{$r$}{}-dimensional simple IAC algebras}\label{sec:rdim}

In this section we construct an infinite family of finite-dimensional
simple IAC algebras. We will work under the following hypothesis.

\begin{hypothesis}\label{hypothesis}
  Fix an integer $b\geq 2$. Let $n>1$ be a divisor of $b^2+b-1$.
  Since $b(b+1)\equiv 1\pmod n$, $b$ is invertible modulo $n$, and
  let~$r$
be the multiplicative order of $b$ modulo $n$. Suppose that $r>1$ and
let $q$ be a prime power such that $n\mid(q-1)$. Let $V=(\F_q)^r$
have basis $e_0,e_1,\dots,e_{r-1}$.
\end{hypothesis}

There are infinitely many choices for a prime-power~$q$
with $q\equiv1\pmod n$ by Dirichlet's Theorem on primes
in an arithmetic progression.

Let $\zeta$ be an $n$-th root of unity in $\F_q$ and let $G$ be the subgroup of $\GL(r,\F)$ generated by the matrices
\[
A=\begin{bmatrix}
0 & 1 & & 0 \\
& & \ddots & \\
0 & 0 & & 1 \\
1 & 0 & \dots & 0\end{bmatrix}
\quad\mbox{and}\quad
B=\begin{bmatrix}
\zeta & 0 & \dots & 0\\
0 & \zeta^b & & 0\\
& & \ddots & \\
0 & 0 & \dots & \zeta^{b^{r-1}}\end{bmatrix}.
\]
Direct calculation shows that $A^r=B^n=1$ and $BA=AB^{b^{-1}}$ where $b^{-1}$
denotes the inverse of $b$ modulo~$n$.
Let $V=(\F_q)^r$, and consider $V$ and its
exterior square $\wedge^2 V$ as $G$-modules.
Since $B$ is a diagonal matrix with pairwise distinct eigenvalues,
and $A$ cyclically permutes the $1$-dimensional eigenspaces of $B$,
we obtain that $G$ acts irreducibly on $V$.

\begin{proposition}\label{prop:galalg}
Assume Hypothesis~\textup{\ref{hypothesis}} and
let
\begin{eqnarray*}
U_1&=&\left\langle e_i\wedge e_j \mid j-i\equiv 1 \bmod r\right \rangle=\left\langle e_0\wedge e_1, e_1\wedge e_2, \dots, e_{r-2}\wedge e_{r-1}, e_{r-1}\wedge e
_0\right \rangle;\\
U_2&=&\left\langle e_i\wedge e_j \mid j-i\not \equiv 1 \bmod r\right\rangle.
\end{eqnarray*}
Then the following are valid:
\begin{enumerate}[{\rm (a)}]
	\item $\wedge^2 V=U_1\oplus U_2$ is a $G$-invariant decomposition;
	\item the map $\psi\colon  V\rightarrow U_1$ defined by
          $e_i\psi=e_{i+1}\wedge e_{i+2}$ reading subscripts modulo~$r$ is a $G$-module isomorphism; and
	\item  $V$ is an ESQ $G$-module.
\end{enumerate}
\end{proposition}
\begin{proof}
  (a) Reduce the indices $i$ modulo~$r$ if $i\ge r$. Then
  $\{e_i\wedge e_{i+1}\mid 0\le i\le r-1\}$ is a basis element of $U_1$. Since
  $(e_i\wedge e_{i+1})A=e_{i+1}\wedge e_{i+2}$, we see that $U_1A=U_1$. A
  similar argument shows that $U_2A=U_2$.
  Since $B$ is a diagonal matrix, we see that
  $\langle e_i\wedge e_j\rangle B=\langle e_i\wedge e_j\rangle$ for all
  $i< j$, and hence $U_1B=U_1$ and $U_2B=U_2$. Thus $\wedge^2 V=U_1\oplus U_2$
  is a $G$-invariant decomposition, as claimed.
  
(b) Note first that $b^2+b\equiv 1\mod n$ implies $\zeta^{b^2+b}=\zeta$ and $\zeta^{b^{i+2}+b^{i+1}}=\zeta^{b^i}$. Let us verify
that $\psi$ commutes with $A$ and with $B$:
\begin{align*} 
	e_iA\psi&=e_{i+1}\psi=e_{i+2}\wedge e_{i+3}=(e_{i+1}\wedge e_{i+2})A=e_i\psi A;\\
	e_iB\psi&=\zeta^{b^i}e_i\psi=\zeta^{b^i}(e_{i+1}\wedge e_{i+2})=\zeta^{b^{i+1}+b^{i+2}}(e_{i+1}\wedge e_{i+2})  \\&=(e_{i+1}\wedge e_{i+2})B=e_i\psi B.  
\end{align*}
Since $\psi$ is surjective and $\dim(V)=\dim(V\psi)=\dim(U_1)$, we obtain that $\psi$ is a bijection.  Thus $\psi$ is an
isomorphism of $G$-modules.

(c) This follows from~(b) and the above fact that $V$ is an irreducible $G$-module.
\end{proof}

Using the ESQ $G$-module $V$ constructed before Proposition~\ref{prop:galalg},
we can construct an
IAC algebra $L$ that corresponds to $V$.

\begin{corollary} 
  Assume Hypothesis~\textup{\ref{hypothesis}}. Let $L$ be the vector space $V=(\F_q)^r$ with an anti-commutative product defined
  on the basis $\{e_0,e_1,\dots,e_{r-1}\}$ by
  \begin{align*}
    \prd{e_0}{e_1}&=e_{r-1},\\
    \prd{e_{i-1}}{e_{i}}&=e_{i-2}\quad\mbox{for } i=2, 3,\ldots,r-1,\\
    \prd{e_{r-1}}{e_0}&=e_{r-2},\\
  \prd{e_i}{e_j}&=0\quad\mbox{\ if $j-i\neq \pm 1 \pmod r$.}
\end{align*}
Then $L$ is an $r$-dimensional simple IAC algebra.
\end{corollary}
\begin{proof}
It follows from Proposition~\ref{prop:galalg}
that $L$ is IAC. Let us show that $L$
is simple. Suppose that $I\lhdeq L$ and let $v\in I\setminus\{0\}$.
Write $v=v_0 e_0+\cdots+v_{r-1}e_{r-1}$ with $v_i\in\F_q$.
Suppose without loss of generality that $v_0\neq 0$. Then using
the multiplication table of $L$, we obtain that
\[
  \prd{\prd{v}{e_1}}{e_0}=\prd{v_0e_{r-1}-v_2e_0}{e_0}=v_0e_{r-2}.
\]
Hence $e_{r-2}\in I$. Multiplying $e_{r-2}$ with the elements $e_{r-1},e_0,\ldots
e_{r-3}$, we obtain that $e_i\in I$ for all $i\in\{0,\ldots,r-1\}$. Thus
$I=L$, and so $L$ is simple as claimed.
\end{proof}

\section{The Clebsch-Gordan formula for the exterior square}\label{sec:CG}

The aim of
Section~\ref{Simple} is to present a class of ESQ modules for the group
$\SL(2,\F)$. In this section we prove a general version  of the
Clebsch--Gordan Formula for the exterior squares of the representations of
$\GL(2,\F)$ on spaces of homogeneous polynomials
that is valid over fields of characteristic $p$.
Such results are usually proved for $\SL(2,\F)$, but generalizing to $\GL(2,\F)$ does not require much more effort.

Let $\F$ be a field of characteristic different from 2 and let $\F[X,Y]$ be the algebra of polynomials over $\F$ in two indeterminates.
An  action of $\GL(2,\F)$ on $\F[X,Y]$ is given by
\begin{equation}\label{E:action}
  \left(\sum \lambda_{ij}X^iY^j\right)A
  =\sum \lambda_{ij}(a_{11}X+a_{12}Y)^i(a_{21}X+a_{22}Y)^j
\end{equation}
where $A=\begin{bmatrix} a_{11}&a_{12}\\a_{21}&a_{22}\end{bmatrix}$.
For each integer $m\ge0$, the subspace of
homogeneous polynomials of degree~$m$ is invariant under this action.
This space has dimension $m+1$ since $\{X^iY^{m-i}\mid 0\le i\le m\}$ is a
basis. When $\Char(\F)=0$ this space is denoted $V_m$, and it
is an irreducible $\F G$-module
for each $m\ge0$. However, if $\F=\F_q$
is a finite field, say with $\Char(\F_q)=p$, we use a different
definition when $p\le m< q$.
In this case, the homogeneous polynomials of degree~$p$ have 
a proper submodule, namely $\langle X^p,Y^p\rangle$. The irreducible
$\textup{SL}(2,\F_q)$-modules in characteristic~$p$ were
determined by Brauer and Nesbitt, and described
clearly in~\cite{BHRD}*{Theorem~5.2.3}. There are $q$ pairwise
non-isomorphic and absolutely irreducible $\F G$-modules which we denote
$V_0,V_1,\dots,V_{q-1}$. If $m<p$, then $V_m$ denotes the subspace of
homogeneous polynomials of degree~$m$ as above. If $p\le m<q$, the
definition of $V_m$ depends on the $p$-adic expansion
$m=m_0+m_1p+\cdots+m_kp^k$ of $m$ where $m_i\in\{0,1,\dots,p-1\}$ and
$k:=\lfloor \log_p m\rfloor$. We set
$V_m=V_{m_0}\otimes V_{m_1}^\phi\otimes\cdots\otimes V_{m_k}^{\phi^k}$ where
$V^\phi$ denotes the $\F G$-module obtained from $V$ by twisting by
the field automorphism $\phi\colon \F\to \F$ with
$\lambda\phi=\lambda^p$, see~\cite{HB}*{Definition~VII.1.13}.

We are interested in irreducible
ESQ-modules, i.e.\ irreducible modules $V$ which are a quotient of their
exterior square. The decomposition of $V_m\otimes V_n$ as a module over
$\textup{SL}(2,\mathbb{C})$ was described by Clebsch and Gordan,
see~\cite{Kowalski}*{Theorem~2.6.3}.
Our primary interest is in $\GL(2,\F)$-module decompositions
when $\F$ is finite. In the following theorem $V_m$ is the $\GL(2,\F)$-module
of homogeneous polynomials of degree~$m$ and
$V_{\det}$ is
the 1-dimensional $\GL(2,\F)$-module on which $\GL(2,\F)$ acts by the rule $vg=(\det g)v$
for all $v\in V_{\det}$ and $g\in \GL(2,\F)$.
For a $G$-module $V$, we let $S^2V$ denote the symmetric square of $V$.

\begin{theorem}\label{T:CG}
  Let $\F$ be a field of characteristic different from $2$ and let $G=\GL(2,\F)$.
  Then the following
decompositions hold where the summands are irreducible $\F G$-modules.
\begin{enumerate}[{\rm (a)}]
  \item If $m+n<\Char(\F)$ or $\Char(\F)=0$, then 
\[
  V_m\otimes V_n\cong\bigoplus_{i=0}^{\min(m,n)} V_{\det}^{\otimes i}\otimes V_{m+n-2i}.
\]
  \item If $0\le 2m<\Char(\F)$ or $\Char(\F)=0$, then
\[
  \wedge^2 V_m\cong
  \bigoplus_{\begin{smallmatrix} i=1\\ \textup{$i$ odd}\end{smallmatrix}}^m  V_{\det}^{\otimes i}\otimes V_{2m-2i}\quad\textup{and}\quad
  S^2 V_m\cong\bigoplus_{\begin{smallmatrix} i=0\\ \textup{$i$ even}\end{smallmatrix}}^m
  V_{\det}^{\otimes i}\otimes V_{2m-2i}.
\]
\end{enumerate}
\end{theorem}

\begin{proof}
(a)~Assume, without loss of generality, that $m\le n$.
We prove the result by induction on $m$. The case when $m=0$ is
clearly true. We view $\F[X_1,Y_1,X_2,Y_2]$ as an $\F G$-module,
by adapting the action~\eqref{E:action} separately  for
the indeterminates $X_1,\ Y_1$ and for $X_2,\ Y_2$.
Let $V_m(X_1,Y_1)$ be the submodule of $\F[X_1,Y_1]$
comprising the degree~$m$ homogeneous polynomials, and define
$V_n(X_2,Y_2)\le \F[X_2,Y_2]$ similarly. The product set
\[
  V_{m,n}:=V_m(X_1,Y_1)V_n(X_2,Y_2)\]
  is an $\F G$-module.
Further, $V_m(X_1,Y_1)\cong V_m$ and $V_n(X_2,Y_2)\cong V_n$.
The multiplication map
$\psi\colon V_m(X_1,Y_1)\otimes V_n(X_2,Y_2)\to V_{m,n}$
defined by
\[
  (f(X_1,Y_1)\otimes g(X_2,Y_2))\psi= f(X_1,Y_1)g(X_2,Y_2)
\]
is a well-defined surjective $\F G$-module homomorphism as an
easy calculation shows that
\[
  ((f\otimes g)A)\psi=(f\otimes g)\psi A
  \qquad\textup{for all $A\in\GL(2,\F)$.}
\]
Comparing dimensions shows that $\psi$ is an $\F G$-isomorphism, and
$V_m\otimes V_n\cong V_{m,n}$.

We now show $V_{m,n}$ and $V_{m-1,n-1}\oplus V_{m+n}$ are isomorphic
$\F G$-modules. Towards this end we will define
an epimorphism $\pi$ and a monomorphism~$\delta$.

Every element of $V_{m,n}$ has the form
$\sum_{i=0}^m\sum_{j=0}^n\lambda_{ij}X_1^iY_1^{m-i}X_2^jY_2^{n-j}$ where
$\lambda_{ij}\in \F$.
The evaluation homomorphism $X_2\mapsto X_1$,  $Y_2\mapsto Y_1$ gives rise to the
following map
\begin{equation}\label{E:psi}
\pi\colon V_{m,n}\rightarrow V_{m+n}(X_1, Y_1)\colon X_1^iY_1^{m-i}X_2^jY_2^{n-j}\mapsto X_1^{i+j}Y_1^{m+n-i-j}.
\end{equation}
It is straightforward to see that $\pi$ is an $\F G$-epimorphism.
We now show that the $\F G$-submodule $W:=(X_1^mX_2^n) \F G$ of $V_{m,n}$ generated
by $X_1^mX_2^n$ is isomorphic to $V_{m+n}$. If $m+n<\Char(\F)$ or $\Char(\F)=0$,
then $V_{m+n}$ is an irreducible $\F G$-module. Since $W\pi$ is a non-zero
submodule of $V_{m+n}$, we have $W\pi=V_{m+n}$ and $\dim(W)\ge m+n+1$.
Conversely, suppose $g=\begin{pmatrix} a&b\\c&d\end{pmatrix}\in\GL(2,\F)$. The binomial theorem gives
\begin{align*}
  (X_1^mX_2^n) g &= (aX_1+bY_1)^m(aX_2+bY_2)^n\\
  &=\sum_{i=0}^m\binom{m}{i}(aX_1)^i(bY_1)^{m-i}\sum_{j=0}^n\binom{n}{j}(aX_2)^j(bY_2)^{n-j}.
\end{align*}
Collecting terms involving $a^kb^{m+n-k}$, where $k=i+j$, gives
\[
  (X_1^mX_2^n) g = \sum_{k=0}^{m+n}a^kb^{m+n-k} h_k
\]
where $h_k$ is the polynomial
$\sum_{i+j=k} \binom{m}{i}\binom{n}{j}X_1^{i}Y_1^{m-i}X_2^{j}Y_2^{n-j}$.
This shows that
$W$ is spanned by $h_0,h_1,\dots,h_{m+n}$. Therefore $\dim(W)\le m+n+1$.
Hence $\dim(W)= m+n+1$, and~$\pi$ restricted to $W$ gives
an $\F G$-isomorphism $W\to V_{m+n}$.

Note that $\F[X_1,Y_1,X_2,Y_2]$ is an integral domain, and multiplying by a
non-zero element in an integral domain is an injective map. Multiplying
by $r:=X_1Y_2-Y_1X_2$ gives the map
\begin{equation}\label{E:phi}
  \delta\colon V_{\det}\otimes V_{m-1, n-1}\rightarrow V_{m,n},\quad
   1\otimes h(X_1,Y_1,X_2,Y_2)\mapsto r\cdot h(X_1,Y_1,X_2,Y_2).
\end{equation}
Observe first that $rA$ equals
\begin{equation}\label{E:r}
  (a_{11}X_1+a_{12}Y_1)(a_{21}X_2+a_{22}Y_2)-(a_{21}X_1+a_{22}Y_1)(a_{11}X_2+a_{12}Y_2)
  =(\det A) r
\end{equation}
where $A=\begin{pmatrix} a_{11}&a_{12}\\a_{21}&a_{22}\end{pmatrix}$.
It follows from~\eqref{E:r} that $rA=\det(A)r$ for $A\in\GL(2,\F)$, and
hence $\langle r\rangle=V_{\det}$.
Let $h=h(X_1,Y_1,X_2,Y_2)\in V_{m-1,n-1}$.
We show that $\delta$ is an $\F G$-homomorphism:
\begin{align*}
  ((1\otimes h)A)\delta&=(\det(A)(hA))\delta=\det(A)(hA)r\quad\textup{and}\\
 ((1\otimes h)\delta)A&=(h r)A=(hA) \det(A)r.
\end{align*}
The subspaces $\im\delta=(X_1Y_2-Y_1X_2)V_{m-1, n-1}$ and $W=(X_1^mX_2^n) \F G$
of $V_{m, n}$ intersect trivially. Thus $V_{m, n}$ has a submodule
isomorphic to $(V_{\det}\otimes V_{m-1,n-1})\oplus V_{m+n}$. However
\[
  \dim((V_{\det}\otimes V_{m-1,n-1})\oplus V_{m+n})=mn +(m+n+1)
  = (m+1)(n+1)=\dim V_{m,n},
\]
and this implies
\[
  (V_{\det}\otimes V_{m-1,n-1})\oplus V_{m+n}\cong V_{m,n}\cong V_m\otimes V_n.
\]
The decomposition (a) now follows by induction on~$m$. Observe that
when $\Char(\F)=p$, then hypothesis $m+n<p$ clearly implies that
$(m-1)+(n-1)<p$.

(b)~Denote the symmetric and exterior squares of $V_m$ by $S^2 V_m$
and $\wedge^2 V_m$, respectively. Since $\Char(\F)\ne2$ we have
$V_m\otimes V_m=S^2 V_m\oplus \wedge^2 V_m$. Our proof uses induction on~$m$.
The formulas are true for $m=0$ because $\wedge^2 V_0=0$ (here the sum
is empty), and $S^2 V_0=V_0$. Suppose $m=1$. A calculation
similar to~\eqref{E:r} shows that
$(X\otimes Y-Y\otimes X)A=\det(A)(X\otimes Y-Y\otimes X)$, and hence
\[
  \wedge^2 V_1=\langle X\wedge Y\rangle\cong V_{\det},
  \quad\textup{and}\quad
  S^2 V_1=\langle X\odot X, X\odot Y, Y\odot Y\rangle\cong V_2,
\]
where $u\wedge v:=\frac12(u\otimes v-v\otimes u)$ and
$u\odot v:=\frac12(u\otimes v+v\otimes u)$. Thus the stated decompositions of
$\wedge^2 V_m$ and $S^2 V_m$ are valid for $m=1$. Suppose now that $m\ge2$.

Recall the definitions in part~(a) of the monomorphism $\delta$,
and the epimorphism~$\pi$.
Let $f(X_1, Y_1)\in V_m(X_1,Y_1)$ and $g(X_2, Y_2)\in V_m(X_2,Y_2)$ be arbitrary.
Then
\[
  V_{m,m}  = V_m(X_1,Y_1)V_m(X_2,Y_2)\cong
  V_m\otimes V_m\cong \wedge^2V_m\oplus S^2V_m
\]
where 
\begin{align}
 \wedge^2V_m&=\left\langle f(X_1, Y_1)g(X_2, Y_2)-g(X_1, Y_1)f(X_2, Y_2)\right\rangle\mbox{, and}\label{E:A}\\
 S^2V_m&=\left\langle f(X_1, Y_1)g(X_2, Y_2)+g(X_1, Y_1)f(X_2, Y_2)\right\rangle.\label{E:S}
\end{align}
The following calculation may be used to show that
$(S^2 V_{m-1})\delta\le\wedge^2V_m$:
\begin{align*}
&(X_1Y_2-Y_1X_2)(f(X_1, Y_1)g(X_2, Y_2)+g(X_1, Y_1)f(X_2, Y_2))\\
&= f(X_1, Y_1)X_1g(X_2, Y_2)Y_2-g(X_1, Y_1)Y_1f(X_2, Y_2)X_2\\
&- [f(X_1, Y_1)Y_1g(X_2, Y_2)X_2-g(X_1, Y_1)X_1f(X_2, Y_2)Y_2].
\end{align*}
The left side lies in $(S^2 V_{m-1})\delta$ by~\eqref{E:A}, and the right side
lies in $\wedge^2V_m$ by~\eqref{E:S}. Thus
$V_{\det}\otimes(S^2 V_{m-1})\cong (S^2 V_{m-1})\delta\le\wedge^2V_m$.
This containment, however, is an equality because
\[
  \dim (S^2V_{m-1})\delta=\dim (S^2V_{m-1})=\binom{\dim(V_{m-1})+1}{2}=\binom{m+1}{2}=\dim (\wedge^2V_m).
\]
The inductive decomposition for $S^2V_{m-1}$ and the isomorphism
$V_{\det}\otimes(S^2 V_{m-1})\cong\wedge^2 V_m$ together imply the desired
decomposition for $\wedge^2 V_m$.

A similar argument using \eqref{E:A} and~\eqref{E:S} shows that
$(\wedge^2 V_{m-1}){\delta}\leq S^2V_m$:
\begin{align*}
&(X_1Y_2-Y_1X_2)(f(X_1, Y_1)g(X_2, Y_2)-g(X_1, Y_1)f(X_2, Y_2))\\
&= f(X_1, Y_1)X_1g(X_2, Y_2)Y_2+g(X_1, Y_1)Y_1f(X_2, Y_2)X_2\\
&- [f(X_1, Y_1)Y_1g(X_2, Y_2)X_2+g(X_1, Y_1)X_1f(X_2, Y_2)Y_2].
\end{align*}
Thus $V_{\det}\otimes(\wedge^2V_{m-1})\cong(\wedge^2V_{m-1})\delta\le S^2V_m$ 
is an $\F G$-submodule of $S^2V_m$. By part~(a),
$V_{2m}$ is an $\F G$-submodule of $V_m\otimes V_m$. Our
hypotheses imply that $V_{2m}$ is irreducible. Since
$V_{2m}\not\le\wedge^2 V_m$, we see that
$V_{2m}\le S^2 V_m$ and
it intersects $(\wedge^2V_{m-1})\delta\cong V_{\det}\otimes\wedge^2V_{m-1}$ trivially.
Therefore 
$\left(V_{\det}\otimes\wedge^2V_{m-1}\right)\oplus V_{2m}\le S^2 V_m$. Comparing dimensions shows that
$\left(V_{\det}\otimes\wedge^2V_{m-1}\right)\oplus V_{2m}=S^2V_m$. The inductive
decomposition for $\wedge^2V_{m-1}$ implies the desired decomposition for $S^2V_m$. This proves part~(b).
\end{proof}

When $\Char(\F)=p$, the hypothesis $m+n<p$ in Theorem~\ref{T:CG}(a)
is necessary; essentially because
$\langle X^p,Y^p\rangle$ is a proper submodule of the homogeneous
polynomials of degree $p$.
Nevertheless, it is possible to relax the hypothesis $m+n<p$
when the number of carries when adding $m$ to $n$ in base-$p$ is zero.

\begin{corollary}
Suppose $\F_q$ has characteristic~$p$ where $q=p^e$.
Let $m,n$ be integers with $0\le m,n<q$ and with $p$-adic expansions
$m=\sum_{j\ge0}m_ip^j$ and $n=\sum_{j\ge0}n_ip^j$. If $m_j+n_j<p$ for
each $j\ge0$, then the following $\GL(2,\F_q)$-decomposition holds
where the modules are twisted by powers of the automorphism
$\lambda^\phi=\lambda^p$, as per~\cite{HB}*{Definition~VII.1.13}:
\[
  V_m\otimes V_n=\bigotimes_{j\ge0}\bigoplus_{i=0}^{\min(m_j,n_j)}
  \left(V_{\det}^{\otimes i}\otimes V_{m_j+n_j-2i}\right)^{\phi^j}.
\]
\end{corollary}

\begin{proof}
This follows from $V_m\otimes V_n=\bigotimes_{j\ge0}V_{m_j}^{\phi^j}\otimes \bigotimes_{j\ge0}V_{n_j}^{\phi^j}=\bigotimes_{j\ge0}(V_{m_j}\otimes V_{n_j})^{\phi^j}$, and Theorem~\ref{T:CG}(a).
\end{proof}

\section{Simple IAC algebras associated to \texorpdfstring{$\SL(2,\F)$}{}}
\label{Simple}
Since $V_0$ is a trivial $\SL(2,\F)$-module,
Theorem~\ref{T:CG}(b) may be used to determine when $V_m$ is an ESQ-module.
We need $2m-2i=m$ to have a solution with $i$ odd. Thus $i=m/2$ and
$m\equiv2\pmod4$.

\begin{theorem}\label{T:GammaL}
Let $V_m$ be the above $\SL(2,\F)$-module where $m\equiv2\pmod4$
and where the field $\F$ satisfies $\Char(\F)=0$ or $2m<\Char(\F)$.
Then $V_m$ is an absolutely irreducible ESQ-module.
Further, if $|\F|\ge5$, the corresponding IAC algebra $L$ is simple.
\end{theorem}

\begin{proof}
We know that $V_m$ is an absolutely irreducible $\SL(2,\F)$-module both when
$\Char(\F)$ is zero or prime, c.f. \cite{BHRD}*{Theorem~5.2.3},
\cite{Kowalski}*{Theorem~2.6.3} and~\cite{thomas}*{pp.\;57--59}.
The preamble to this theorem showed that $V_m$ is a quotient of $\wedge^2 V_m$.
Thus we may turn $L=V_m$ into an anti-commutative algebra 
satisfying $\prd{L}{L}=L$, as in Section~\ref{sec1}.
Since $\SL(2,\F)$ acts irreducibly on $L=V_m$, we see that $L$ is an IAC
algebra.

It remains to prove that $L$ is a simple IAC algebra. 
By Theorem~\ref{SimpleSum}(b), $L=I_1\oplus\cdots\oplus I_\ell$ where
$I_1,\dots,I_\ell$ are the minimal ideals of $L$ and $\ell\ge1$.
The group $\SL(2,\F)$ permutes
the set $X=\{I_1,\dots,I_\ell\}$ transitively. Since $Z=\Z(\SL(2,\F))$
acts as the identity on this set, the projective linear group $\PSL(2,\F)$
acts on $X$. Further, $L\cong I_1^{\oplus \ell}$ where $I_1$
is a simple IAC algebra, and so $\dim I_1\geq 3$, which implies
that $\ell=\dim (V_m)/\dim(I_1)\le(m+1)/3$.

It is well-known that $\PSL(2,\F)$ is a simple group for $|\F|\ge5$.
Suppose $\ell>1$.
The homomorphism $\PSL(2,\F)\to\Sym_\ell$ is faithful, and the image is
transitive on $\ell$ points.
This is impossible if $\F$ is infinite, as $\Sym_\ell$ is finite and $\PSL(2,\F)$
is not. Thus $|\F|=q$ is finite. Set $p=\Char(\F)$. The minimal degree $d$ of
a transitive permutation representation of $\PSL(2,\F_q)$ is $q+1$
except for $q=5,7,9,11$ in these cases $d=5,7,6,11$,
respectively~\cite{KL}*{Table~5.2.A, p.\;175}.
In all the cases we have $2q/3\le d$.
Our hypothesis $2m<p$ implies
\[
  \frac{2p}{3}\le \frac{2q}{3}\le d\le \ell\le\frac{m+1}{3}<\frac{p/2+1}{3}.
\]
This is false for all primes~$p$, so we conclude that $\ell=1$, and $L$ is simple.
\end{proof}

\begin{remark}
  Theorem~\ref{T:GammaL} can be generalized to include ESQ $H$-modules where
  $\SL(2,\F_q)<H\le\GL(2,\F_q)$. For example, if $m\equiv 2\pmod4$ and
  $i=m/2$ is a multiple of $|H:\SL(2,\F_q)|$, then $V_{\det}^{\otimes i}$ is a
  trivial $H$-module and $L=V_m$ can be turned into an
  IAC algebra with $H\le\Aut(L)$.
\end{remark}

Combining the special case of $\F=\F_p$ in
Theorem~\ref{T:GammaL} with the bijection in Theorem~\ref{th:bij}
gives the following corollary stating the existence of another infinite
family of UCS $p$-groups of exponent $p^2$.

\begin{corollary}
  Let $p\geq 3$ be a prime, and let $m$ be a natural number such that $2m<p$.
  Then there exists an $(m+1)$-generator UCS $p$-group $G$ with exponent $p^2$
  such that the
  subgroup induced by $\Aut(G)$ on $\overline G$ contains $\SL(2,p)$ acting
  in its absolutely irreducible representation of dimension $m+1$ over
  the field $\F_p$. Further, $G$ cannot be written as a direct product
  $(G_0)^{\times k}$ where $G_0$ is a smaller UCS $p$-group.
\end{corollary}

\section*{Acknowledgments}

The first author gratefully acknowledges
the support of the Australian Research Council Discovery Grant DP160102323.
The third author thanks CNPq for Bolsa de Produtividade em Pesquisa 
Project 308773/2016-0, and Universal Project 475399/2013-7,
as well as acknowledging the hospitality and financial support of the
Centre for the Mathematics of Symmetry and Computation (CMSC) of UWA
during visits in 2014 and 2017.
The second author is grateful to Fapemig for a PhD scholarship and
financially supporting his visit to the CMSC. We thank the referee for
helpful suggestions.

\end{document}